\documentclass[a4paper, 11pt]{amsart}   
\usepackage{mathptmx, amssymb,amscd,latexsym, eulervm}   
\usepackage{amsmath}
\usepackage{amsthm}
\usepackage{mathdots}
\usepackage[pagebackref,colorlinks=true,citecolor=violet,linkcolor=blue,urlcolor=blue]{hyperref}
\usepackage[dvipsnames]{xcolor}
\usepackage[onehalfspacing]{setspace}
\usepackage{tabularx}
\usepackage{amsfonts}
\usepackage{paralist}
\usepackage{aliascnt}
\usepackage{amscd}
\usepackage{blkarray}
\usepackage{mathbbol}
\usepackage{setspace}
\usepackage[inner=2.4cm,outer=2.4cm, bottom=3.2cm]{geometry}
\usepackage{tikz, tikz-cd}
\usepackage{calligra,mathrsfs}

\usepackage{tikz}
\usetikzlibrary{matrix}
\usetikzlibrary{arrows,calc}
\allowdisplaybreaks

\AtBeginDocument{%
	\def\MR#1{}
}

\makeatletter
\@namedef{subjclassname@2020}{%
	\textup{2020} Mathematics Subject Classification}
\makeatother

\newcommand{\Bl}{{\mathcal{B}\ell}}

\newcommand{\sM}{\mathscr{M}}
\newcommand{\sC}{\mathcal{C}}
\newcommand{\sD}{\mathcal{D}}
\newcommand{\sY}{\mathscr{Y}}
\newcommand{\sW}{\mathscr{W}}
\newcommand{\cB}{\mathcal{B}}
\newcommand{\kk}{\mathbb{k}}

\newcommand{\NN}{\normalfont\mathbb{N}}
\newcommand{\ZZ}{\mathbb{Z}}

\newcommand{\PP}{{\normalfont\mathbb{P}}}

\newcommand{\mm}{{\normalfont\mathfrak{m}}}

\newcommand{\QQ}{\mathbb{Q}}

\newcommand{\aaa}{\mathcal{A}}

\newcommand{\length}{{\rm length}}

\newcommand{\II}{\mathscr{I}}
\newcommand{\JJ}{\mathscr{J}}

\newcommand{\Sym}{\normalfont\text{Sym}}
\newcommand{\Rees}{\mathscr{R}}

\newcommand{\OO}{\mathcal{O}}

\newcommand{\LL}{\mathscr{L}}
\newcommand{\llf}{\mathbb{L}}

\newcommand{\HH}{\normalfont\text{H}}

\newcommand{\gr}{{\normalfont\text{gr}}}

\newcommand{\Proj}{\normalfont\text{Proj}}

\newcommand{\Spec}{\normalfont\text{Spec}}

\def\f0{\mathbf{0}}

\def\1{\mathbf{1}}

\newtheorem{theorem}{Theorem}[section]

\newtheorem{headthm}{Theorem}

\newaliascnt{headcor}{headthm}
\newtheorem{headcor}[headcor]{Corollary}
\aliascntresetthe{headcor}

\newaliascnt{headconj}{headthm}

\aliascntresetthe{headconj}

\newaliascnt{corollary}{theorem}
\newtheorem{corollary}[corollary]{Corollary}
\aliascntresetthe{corollary}

\newaliascnt{claim}{theorem}
\newtheorem{claim}[claim]{Claim}
\aliascntresetthe{claim}

\newaliascnt{lemma}{theorem}
\newtheorem{lemma}[lemma]{Lemma}
\aliascntresetthe{lemma}

\newaliascnt{conjecture}{theorem}

\aliascntresetthe{conjecture}

\newaliascnt{proposition}{theorem}
\newtheorem{proposition}[proposition]{Proposition}
\aliascntresetthe{proposition}

\theoremstyle{definition}
\newaliascnt{definition}{theorem}
\newtheorem{definition}[definition]{Definition}
\aliascntresetthe{definition}

\newaliascnt{notation}{theorem}

\aliascntresetthe{notation}

\newaliascnt{example}{theorem}
\newtheorem{example}[example]{Example}
\aliascntresetthe{example}

\newaliascnt{examples}{theorem}

\aliascntresetthe{examples}

\newaliascnt{remark}{theorem}
\newtheorem{remark}[remark]{Remark}
\aliascntresetthe{remark}

\newaliascnt{question}{theorem}

\aliascntresetthe{question}

\newaliascnt{questions}{theorem}

\aliascntresetthe{questions}

\newaliascnt{problem}{theorem}

\aliascntresetthe{problem}

\newaliascnt{construction}{theorem}

\aliascntresetthe{construction}

\newaliascnt{setup}{theorem}
\newtheorem{setup}[setup]{Setup}
\aliascntresetthe{setup}

\newaliascnt{algorithm}{theorem}

\aliascntresetthe{algorithm}

\newaliascnt{observation}{theorem}

\aliascntresetthe{observation}

\newaliascnt{defprop}{theorem}

\aliascntresetthe{defprop}

\newaliascnt{fact}{theorem}

\aliascntresetthe{fact}

\newtheorem{chunk}[definition]{}

\DeclareFontFamily{OT1}{pzc}{}
\DeclareFontShape{OT1}{pzc}{m}{it}{<-> s * [1.100] pzcmi7t}{}
\DeclareMathAlphabet{\mathchanc}{OT1}{pzc}{m}{it}

\DeclareMathOperator{\fSpec}{\mathchanc{Spec}}
\DeclareMathOperator{\fProj}{\mathchanc{Proj}}

\def\equationautorefname~#1\null{(#1)\null}
\def\sectionautorefname~#1\null{Section #1\null}
\def\subsectionautorefname~#1\null{\S #1\null}

\def\surjects{\twoheadrightarrow}

\begin{document}

	\title{Segre classes and integral dependence}

	\author{Yairon Cid-Ruiz}
	\address{Department of Mathematics, North Carolina State University, Raleigh, NC 27695, USA}
	\email{ycidrui@ncsu.edu}


	\begin{abstract}
		A fundamental property of Segre classes is their birational invariance.
		This invariance implies that the Segre class of a closed subscheme only depends on the integral closure of the defining ideal sheaf.
		
		In this paper, we show that, conversely, the Segre class of a closed subscheme encodes an integral dependence criterion for its defining ideal sheaf.
		As an application, we prove that Aluffi’s Segre zeta function provides an integral dependence criterion for homogeneous ideals in polynomial rings.
	\end{abstract}

\maketitle

\section{Introduction}

The \emph{Segre class} associated with an embedding of schemes plays a central role in Fulton--MacPherson's intersection theory and underlies numerous applications.
A detailed account of Segre classes and their applications can be found in \cite{FULTON_INTER,ALUFFI_SURVEY}.

A central and defining property of Segre classes is their \emph{birational invariance}.
Indeed, if $Z$ is a proper closed subvariety of a variety $X$ over a field $\kk$, this invariance yields 
$$
s(Z, X) \;=\; \eta_*\left(s(E, \mathcal{B})\right) \;=\; \eta_*\left(c\left(N_E\mathcal{B}\right)^{-1} \smallfrown [E]\right) \;=\; \eta_*\left(\frac{[E]}{1+E}\right) \;\in\; A_*(Z),
$$
where $\pi : \mathcal{B} = \Bl_Z(X) \rightarrow X$ is the blow-up of $X$ along $Z$, $\eta : E = E_ZX \rightarrow Z$ is the exceptional divisor, and $N_E\mathcal{B}$ is the normal bundle to $E$ in $\mathcal{B}$ (see \cite[\S 4.2]{FULTON_INTER}).
Thus, by birational invariance, computing Segre classes reduces to the case of effective Cartier divisors, whose Segre classes admit an explicit formula since they are regularly embedded.

The birational invariance of Segre classes also implies that they depend only on the \emph{integral closure} of the underlying ideal sheaf.
More precisely, if we let $Z'= V\left(\overline{\II_Z}\right) \subset X$ where $\overline{\II_Z} \subset \OO_X$ is the integral closure of the ideal sheaf $\II_Z$ of $Z$, then the Segre class $s(Z, X) \in A_*(Z)$ of $Z$ equals the Segre class $s(Z', X) \in A_*(Z')$ of $Z'$ (since $Z'_{\rm red} = Z_{\rm red}$, we can identify their Chow groups).

\smallskip

Our main result demonstrates that Segre classes, in fact, determine integral dependence, leading to the following criterion:
	
\begin{headthm}[\autoref{thm_main}]
	\label{thmA}
	Let $X$ be an equidimensional projective scheme over a field $\kk$.
	Let $\LL$ be an ample line bundle on $X$ and $W \subseteq Z$ be two closed subschemes of $X$.
	Then the following four conditions are equivalent:
	\begin{enumerate}[\;\;\rm (a)]
		\item $\II_W$ is integral over $\II_Z$.
		\item[\;\;\rm(b$'$)] $s(Z, X) = s(W, X)$ viewed in $A_*(Z)$.
		\item $s(Z, X) = s(W, X)$ viewed in $A_*(X)$.
		\item $\deg_\LL\left(s^i(Z, X)\right) = \deg_\LL\left(s^i(W, X)\right)$ for all $i \ge 0$.
	\end{enumerate} 
\end{headthm}		

In \autoref{thmA}, even if one is not interested in integral dependence, the equivalence (b) $\Leftrightarrow$ (c) seems interesting in its own right.  
We note that \autoref{thmA} is sharp in the sense that the ampleness hypothesis appears to be indispensable. 
\autoref{examp} illustrates this by exhibiting a big and nef line bundle for which the conclusion of \autoref{thmA} does not hold.

The proof of \autoref{thmA} relies on van Gastel’s foundational result \cite{VanGastel}, which establishes the connection between Fulton–MacPherson’s intersection theory \cite{FULTON_INTER} and St\"uckrad–Vogel’s intersection theory \cite{VOGEL,FOV}. 
Indeed, our starting point is van Gastel’s result expressing Segre classes in terms of Vogel cycles, and conversely expressing Vogel cycles in terms of Segre classes.
An advantage of considering the Vogel cycle is that it opens the door to applying technical positivity results, including the one established in \autoref{prop_hard}.

\smallskip

As an application of \autoref{thmA}, we show that Aluffi’s Segre zeta function \cite{aluffi2017segre} furnishes a criterion for the integral dependence of homogeneous ideals in a polynomial ring. 
This power series records information about the Segre classes that arise when the ideal is extended to projective spaces of arbitrarily large dimension (see \autoref{sect_zeta} for further details). 
We obtain the following corollary:

\begin{headcor}[\autoref{thm_main_zeta}]
	\label{corB}
	Let $I \subseteq J \subset R$ be two homogeneous ideals in a polynomial ring $R = \kk[x_0, \ldots,x_n]$.
	Then the following two conditions are equivalent:
	\begin{enumerate}[\;\;\rm (a)]
		\item $J$ is integral over $I$.
		\item $\zeta_I(t) = \zeta_J(t)$.
	\end{enumerate} 
\end{headcor}

The idea of detecting integral dependence through numerical invariants originates in Rees’ seminal work \cite{REES}.
Rees proved that in an equidimensional and universally catenary Noetherian local ring $(R, \mm)$, two $\mm$-primary ideals $I \subseteq J$ have the same integral closure if and only if they share the same Hilbert–Samuel multiplicity.
Since then, the search for numerical criteria to characterize integral dependence has been an important research topic in algebraic geometry, commutative algebra and singularity
theory.
Much effort has gone into extending Rees’ theorem to arbitrary ideals, modules, and more broadly, to algebras (see \cite{BOEGER,TEISSIER_CYC,TEISSIER_RES2, KLEIMAN_THORUP_GEOM,KLEIMAN_THORUP_MIXED, GG, FLENNER_MANARESI,SUV_MULT, GAFFNEY, CIU, UV_CRIT_MOD,     UV_NUM_CRIT,   PTUV, cid2023relative, cidruiz2024polar, CRPU2, BLQ, das2024numerical}). 
In this direction, two notable results are: 
\begin{itemize}[\;--]
	\item Teissier's Principle of Specialization of Integral Dependence (PSID) \cite{TEISSIER_CYC}, \cite[Appendice I]{TEISSIER_RES2}.
	Given a map of germs of analytic spaces,  the PSID asserts that
	for a family of zero-dimensional ideals with constant Hilbert-Samuel multiplicity, a section is integrally
	dependent on the total family if and only if it is integrally dependent along the fibers corresponding to a dense open  subset of the base.
	\item An extension of Rees’ theorem to arbitrary ideals (in an equidimensional and universally catenary Noetherian local ring $(R, \mm)$) was established by Polini, Trung, Ulrich, and Validashti \cite{PTUV}, using the multiplicity sequence introduced by Achilles and Manaresi \cite{ACH_MANA}. (In the analytic setting, an analogous result had been obtained earlier by Gaffney and Gassler \cite{GG}.)
\end{itemize}
Note that these criteria are local in nature, whereas \autoref{thmA} furnishes a global criterion.

\medskip
 
  \noindent
\textbf{Outline.} 
The structure of the paper is as follows. 
In \autoref{sect_prelim}, we introduce the notation used throughout and recall the necessary preliminary results. 
\autoref{sect_van_Gastel} develops the relation between Segre classes and Vogel cycles and shows how this connection can be applied. 
The proof of \autoref{thmA} is given in \autoref{sect_main}. 
Finally, \autoref{sect_zeta} contains the proof of \autoref{corB}.

\section{Preliminaries and Notation}
\label{sect_prelim}

In this section, we set up the notation that is used throughout the paper. 
We also present some preliminary results regarding integral closure (for more details, the reader is referred to \cite{SwHu}, \cite{LJT}, \cite{VASC_INT}, \cite[\S 9.6.A]{LAZARSFELD2}).

\begin{chunk}
	We use the convention that $\binom{m}{-1}:=0$ for $m \ge 0$ and $\binom{-1}{-1}:=1$.
	This convention is used throughout this paper, including \autoref{thm_van_Gastel}.
\end{chunk}

\begin{chunk}
	Let $X$ be a proper scheme over  a field $\kk$.
	For any coherent sheaf $\mathcal{F}$ on $X$, we use the notation 
	$$
	h^i(X, \mathcal{F}) \;:=\; \dim_\kk\left(\HH^i(X, \mathcal{F})\right) \quad \text{ and } \quad \chi(\mathcal{F}) \;:=\; \sum_{i\ge 0} (-1)^i h^i(X, \mathcal{F}).
	$$
\end{chunk}

\begin{chunk}
	Let $X$ be a proper scheme over a field $\kk$.
	Let $\LL$  be a line bundle on $X$ and $\alpha \in A_d(X)$ be a class of pure dimension $d$. 
	The \emph{$\LL$-degree} of $\alpha$ is given by the intersection number
	$$
	\deg_\LL\left(\alpha\right) \;:=\; \int c_1(\LL)^d \smallfrown \alpha.
	$$
	For all $p \ge 1$, we have the equality $\deg_{\LL^{\otimes p}}(\alpha) = p^d \deg_\LL(\alpha)$.
\end{chunk}

\begin{chunk}
	We  recall some useful notation from \cite{KLEIMAN_THORUP_GEOM,KLEIMAN_THORUP_MIXED}.
	Let $X$ be a scheme, $W$ be a subscheme of $X$  and $\mathbf{S} \in Z_*(X)$ be a cycle on $X$. 
	We define the $W$-part $\mathbf{S}^W$ as follows: if $\mathbf{S}$ is the fundamental cycle of an integral closed subscheme $S \subset X$ of $X$, define $\mathbf{S}^W$ to be $\mathbf{S}$ if $W$ contains the generic point of $S$, and to be zero if not; then extend this definition by linearity.
\end{chunk}

\begin{chunk}
	Let $I$ be an ideal in a commutative ring $A$.
	The \emph{integral closure} $\overline{I} \subset A$ of $I$ is the ideal given by 
	$$
	\overline{\,I\,} \;:=\; \Big\lbrace
	\begin{array}{c|c}
		f \in A & f^n + a_1f^{n-1} + \cdots + a_n = 0 \;\text{ where $n\ge 1$ and $a_i \in I^i$ for $1 \le i \le n$}
	\end{array}
	  \Big\rbrace.
	$$
\end{chunk}

\begin{chunk}
	Let $X$ be a scheme and $\II \subset \OO_X$ be an ideal sheaf in $X$. 
	The \emph{integral closure} $\overline{\II} \subset \OO_X$ of $\II$ is the sheaf associated to the presheaf
	$$
	U \subset X \;\mapsto\; \overline{\II(U)} \subset \OO_X(U).
	$$
\end{chunk}

\begin{chunk}
	Let $X$ be a scheme and $\II \subset \OO_X$ be an ideal sheaf in $X$. 
	The \emph{Rees algebra} of $\II$ is given by 
	$$
	\Rees(\II) \;:=\; \bigoplus_{n \ge 0} \II^n.
	$$
	The \emph{blow-up} of $X$ along a closed subscheme $Z \subset X$ is given by $\Bl_Z(X) := \fProj_X\left(\Rees(\II_Z)\right)$. 
\end{chunk}

\begin{chunk}
	\label{chunk_int_dep}
	Let $X$ be a Noetherian scheme and $\II \subseteq \JJ \subset \OO_X$ be two ideal sheaves in $X$.
	We say that \emph{$\JJ$ is integral over $\II$} if one of the following equivalent conditions is satisfied:
	\begin{enumerate}[\rm \;\; (a)]
		\item $\overline{\II} = \overline{\JJ}$.
		\item $\II \cdot  \JJ^n = \JJ^{n+1}$ for all $n \gg 0$.
		\item  $\Rees\left(\JJ\right)$ is of finite type as a module over $\Rees\left(\II\right)$.
		\item The inclusion $\Rees(\II) \hookrightarrow \Rees(\JJ)$ induces a finite morphism $\Bl_W(X) \rightarrow \Bl_Z(X)$ of blow-ups, where $Z = V(\II) \subset X$ and $W = V(\JJ) \subset X$ are the corresponding closed subschemes of $X$. 
		\item $\fProj_X\Big(\Rees(\JJ)\big/\Rees(\II)_+\Rees(\JJ)\Big)=\varnothing$, where $\Rees(\II)_+ := \bigoplus_{n \ge 1} \II^n$.
	\end{enumerate}
\end{chunk}

\section{Segre classes, Vogel cycles, and polar cycles}
\label{sect_van_Gastel}

This section establishes the connection between Segre classes, Vogel cycles, and polar cycles. 
Our main tool will be the fundamental result of van Gastel  \cite{VanGastel} expressing Segre classes in terms of Vogel cycles and vice versa (also, see \cite[Chapter 2]{FOV}).
For any unexplained notation regarding intersection theory, the reader is referred to \cite{FULTON_INTER}.

\begin{setup}
	\label{setup_translations}
	Let $\kk$ be a field, $X$ be an equidimensional proper scheme over $\kk$, and $\LL$ be a line bundle on $X$.
	Let $d := \dim(X)$ be the dimension of $X$.
	Let $r \ge 1$ and $\underline{\sigma} = \sigma_1,\ldots,\sigma_r \in \HH^0(X, \LL)$ be global sections of $\LL$.
	Let $Z := V(\sigma_1) \cap \cdots \cap V(\sigma_r) \subset X$ be the closed subscheme given by intersecting the sets of zeros of the sections $\underline{\sigma}$.
\end{setup}

Consider the purely transcendental field extension 
$$
\llf \;:=\;  \kk\left(u_{i,j} \,\mid\,  1 \le i, j \le r\right)
$$
where $u_{i,j}$ are new indeterminates.
We write $X_\llf := X \times_{\Spec(\kk)} \Spec(\llf)$ and $Z_\llf := Z \times_{\Spec(\kk)} \Spec(\llf)$.
We now discuss the intersection algorithm that leads to the construction of the Vogel cycle.

\begin{definition}[{Intersection Algorithm; see \cite{VanGastel}, \cite[Chapter 2]{FOV}}]
	\label{def_inter_algo}
	Let $D_i := V(s_i)$ be the  set of zeros of the generic linear combination
	$$
	s_i \;:=\; \sum_{j=1}^r u_{i,j} \sigma_j.
	$$
	Inductively, we define two sequences $\nu^i(\underline{\sigma}, X)$ and $\beta^i(\underline{\sigma}, X)$ of cycles for $0 \le i \le r$.
	\begin{itemize}
		\item[\;\;-- (Initial Step):] Decompose $[X_\llf] = \nu^0(\underline{\sigma}, X) + \beta^0(\underline{\sigma}, X)$, where $\nu^0(\underline{\sigma}, X):=\left[X_\llf\right]^{Z_\llf}$ is the part of the fundamental cycle $[X_\llf]$ supported by $Z_\llf$, and $\beta^0(\underline{\sigma}, X):=\left[X_\llf\right]^{X_\llf\setminus Z_\llf}$ is the rest.
		\item[\;\;-- (Inductive Step):]
		Assume that $\nu^{i-1}(\underline{\sigma}, X)$ and $\beta^{i-1}(\underline{\sigma}, X)$ have been computed.
		We obtain that $D_i$ restricts to an effective Cartier divisor on $\beta^{i-1}(\underline{\sigma}, X)$.
		Consider the intersection of $\beta^{i-1}(\underline{\sigma}, X)$ with the divisor $D_i$.
		Decompose this intersection as
		\begin{equation}
			\label{eq_decomp_cycles}
			D_i \,\cdot\, \beta^{i-1}(\underline{\sigma}, X) \;=\; \nu^i(\underline{\sigma}, X) + \beta^i(\underline{\sigma}, X),
		\end{equation}
		 where 
		 $$
		 \nu^i(\underline{\sigma}, X) \;:=\; \left[D_i \,\cdot\, \beta^{i-1}(\underline{\sigma}, X)\right]^{Z_\llf}
		 $$ 
		 is the part of the intersection supported by $Z_\llf$, and 
		 $$
		 \beta^i(\underline{\sigma}, X) \;:=\; \left[D_i \,\cdot\, \beta^{i-1}(\underline{\sigma}, X)\right]^{X_\llf\setminus Z_\llf}
		 $$ 
		 is the rest.
	\end{itemize}
	We have that $\beta^r(\underline{\sigma}, X) = 0$, so we set $\nu^i(\underline{\sigma}, X) := 0$ and $\beta^i(\underline{\sigma}, X) := 0$ for $i > r$.
	The cycles $\nu^i(\underline{\sigma}, X)$ and $\beta^i(\underline{\sigma}, X)$ (if nonzero) are effective cycles of pure dimension $d-i$ on $Z_\llf$ and on $X_\llf$, respectively.
\end{definition}

\begin{definition}[{Vogel cycle; see \cite{VanGastel}, \cite[Chapter 2]{FOV}}]
	\label{def_vogel_cycle}
	The \emph{Vogel cycle} of $\underline{\sigma}$ on $X$ is the cycle 
	$$
	\nu(\underline{\sigma}, X) \;:=\; \sum_{i\ge 0}\nu^i(\underline{\sigma}, X) \;\;\in\;\; Z_*\left(Z_\llf\right).
	$$ 
\end{definition}

\begin{definition}[Polar cycle]
	The \emph{polar cycle} of $\underline{\sigma}$ on $X$ is the cycle 
	$$
	\beta(\underline{\sigma}, X) \;:=\; \sum_{i\ge 0}\beta^i(\underline{\sigma}, X) \;\in\; Z_*\left(X_\llf\right).
	$$
\end{definition}

The reason for calling $\beta(\underline{\sigma}, X)$ the polar cycle will become clear after \autoref{prop_blow_up_formula}.

\begin{remark}
	\label{rem_base_change}
By a result of Fulton (see \cite[Lemma 1.4]{VanGastel}, \cite[Proposition 2.1.8]{FOV}), we have a natural isomorphism 
$$
A_*(Y)  \;\xrightarrow{\;\cong\;} \; A_*\left(Y \times_{\Spec(\kk)} \Spec(\llf)\right), \qquad \left[V\right] \mapsto \left[V \times_{\Spec(\kk)} \Spec(\llf)\right],
$$
for any $\kk$-scheme $Y$ of finite type.
\end{remark}

Hence, under this natural isomorphism, we can make the identifications
$$
\nu(\underline{\sigma}, X) \;\in\; A_*(Z) \quad \text{ and } \quad \beta(\underline{\sigma}, X) \;\in\; A_*(X) 
$$
(i.e., we can consider rational equivalence classes over $\kk$ instead of  $\llf$).

Let $C:=C_ZX := \fSpec_Z(\gr_{\II_Z}(\OO_X))$ be the normal cone to $Z$ in $X$, where 
$$
\gr_{\II_Z}(\OO_X) \;:=\; \bigoplus_{n\ge 0} \II_Z^n/\II_Z^{n+1}
$$ 
is the associated graded ring of $\II_Z$.

\begin{definition}
	\label{def_Segre_class}
	Let $q : \PP\left(C \oplus 1\right) \rightarrow Z$ be the projective completion of the cone $C$.
	The \emph{$i$-th Segre class} of $Z$ in $X$ is given by 
	$$
	s^i(Z, X) \;:=\; s^i(C) \;:=\; q_*\Big(c_1(\OO_{\PP\left(C \oplus 1\right)}(1))^i \smallfrown \left[\PP\left(C \oplus 1\right)\right]\Big)
	$$ 
	where $\OO_{\PP\left(C \oplus 1\right)}(1)$ is the canonical line bundle on $\PP\left(C \oplus 1\right)$.
	Then 
	$$
	s(Z, X) \;:=\; \sum_{i\ge 0} s^i(Z, X)
	$$ 
	is the (total) \emph{Segre class} of $Z$ in $X$.
\end{definition}

\begin{remark}
	Since $X$ is equidimensional of dimension $d$, we get that $\PP\left(C \oplus 1\right)$ is also equidimensional of dimension $d$ (see, e.g., \cite[\S B.5]{FULTON_INTER}), and so it follows that $s^i(Z, X) \in A_{d-i}(Z)$.
\end{remark}

We also consider the following twisted version of the normal cone 
$$
C^\LL \;:=\; C_ZX \otimes \LL \;:=\; \fSpec_Z\left(\bigoplus_{n \ge 0} \II_Z^n/\II_Z^{n+1} \otimes \LL^{\otimes n}\right).
$$
An advantage of $C^\LL$ over $C$ is the following: the first graded part $\II_Z/\II_Z^2 \otimes \LL$ of $C^\LL$ is generated by global sections. 
Indeed, by construction $\II_Z \otimes \LL$ is generated by the sections $\sigma_1,\ldots,\sigma_r$, and so the coherent sheaf $\II_Z/\II_Z^2 \otimes \LL$ is generated by the initial forms $\sigma_1^*,\ldots,\sigma_r^* \in \HH^0\big(X, \II_Z/\II_Z^2 \otimes \LL\big)$.

We now present van Gastel's fundamental result relating Segre classes and Vogel cycles.

\begin{theorem}[{van Gastel; see \cite[Corollary 3.7]{VanGastel}, \cite[Corollary 2.4.7]{FOV}}]
	\label{thm_van_Gastel}
	Modulo rational equivalence on $Z$, for all $i \ge 0$, we have the equalities
	$$
	\nu^i(\underline{\sigma}, X) \;=\; s^i\big(C_ZX \otimes \LL\big) \;=\; \sum_{j=0}^i\binom{i-1}{j-1} c_1(\LL)^{i-j} \smallfrown s^j(Z, X) \; \in\; A_{d-i}(Z)
	$$
	and 
	$$
	s^i(Z, X) \;=\; \sum_{j=0}^i \binom{i-1}{j-1} {(-1)}^{i-j} c_1(\LL)^{i-j} \smallfrown \nu^j(\underline{\sigma}, X) \;\in\; A_{d-i}(Z).
	$$
	In particular, the Vogel cycle $\nu(\underline{\sigma}, X)$ and the Segre class $s(Z, X)$ determine each other modulo rational equivalence on $Z$.
\end{theorem}

Let $\mathcal{B} := \Bl_Z(X) = \fProj_X\left(\Rees(\II_Z)\right)$ be the blow-up of $X$ along $Z$, where $\Rees(\II_Z) = \bigoplus_{n\ge 0} \II_Z^n$ is the Rees algebra of the ideal sheaf $\II_Z \subset \OO_X$.
Let $\pi : \mathcal{B} \rightarrow X$ be the natural projection. 
Consider the twisted Rees algebra 
$$
\Rees^\LL(\II_Z) \;:=\; \bigoplus_{n\ge 0} \II_Z^n \otimes \LL^{\otimes n}
$$
and the twisted blow-up $\mathcal{B}^\LL := \Bl_Z^\LL(X) := \fProj_X\left(\Rees^\LL(\II_Z)\right)$.
We have a natural isomorphism $\mathcal{B}^\LL \cong \mathcal{B}$ with the identification 
\begin{equation}
	\label{eq_can_line_O(1)}
	\OO_{\mathcal{B}^\LL}(1) \;\cong\; \OO_{\mathcal{B}}(1) \otimes \pi^*(\LL);
\end{equation}
see, e.g., \cite[Lemma II.7.9]{HARTSHORNE}.

\begin{remark}
	\label{rem_nice_immersion}
	Since $\II_Z \otimes \LL$ is generated by the global sections $\sigma_1,\ldots,\sigma_r$, we obtain a natural surjection 
	$
	\Sym\left(\OO_X^r\right) \;\surjects \; \Rees^\LL(\II_Z)
	$
	that yields the closed immersion 
	$$
	\mathcal{B}^\LL = \fProj_X\left(\Rees^\LL(\II_Z)\right) \;\subset\; \fProj_X\left(\Sym\left(\OO_X^r\right)\right) \;=\; X \times_\kk \PP_{\kk}^{r-1}.
	$$
	Denoting by $p : \mathcal{B}^\LL \subset X \times_\kk \PP_{\kk}^{r-1} \rightarrow \PP_{\kk}^{r-1}$ the natural projection, we get $\OO_{\mathcal{B}^\LL}(1) \cong p^*\big(\OO_{\PP_{\kk}^{r-1}}(1)\big)$.
	We have that $\OO_{\mathcal{B}^\LL}(1)$ is very ample over $X$.
\end{remark}

\begin{remark}
	\label{rem_blowup_irr_comp}
	The irreducible components of $\cB=\Bl_Z(X)$ are mapped (via the projection $\pi$) birationally onto the irreducible components of $X$ not contained in $Z$.
	Consequently, since $X$ is equidimensional of dimension $d$, it follows that either $\cB$ is equidimensional of dimension $d$ or $\cB = \varnothing$ (the latter case happens when $Z_{\rm red} = X_{\rm red}$).
\end{remark}

The following proposition gives a blow-up formula for the polar cycle $\beta(\underline{\sigma}, X)$ (cf., \cite[Lemma 2.2]{GG}, \cite[Theorem C(i)]{CRPU2}).
This blow-up formula is a consequence of \autoref{thm_van_Gastel}.

\begin{proposition}[Blow-up formula]
	\label{prop_blow_up_formula}
	Modulo rational equivalence on $X$, for all $i \ge 0$, we have the equality
	$$
	\beta^i(\underline{\sigma}, X) \;=\; \pi_*\left(c_1\left(\OO_{\mathcal{B}^\LL}(1)\right)^i \smallfrown \left[\Bl_Z(X)\right]\right) \;\in\; A_{d-i}(X). 
	$$
\end{proposition}
\begin{proof}
	We proceed by induction of $i$.
	By \autoref{rem_blowup_irr_comp}, it follows that $\pi_*\left(\left[\cB\right]\right) = [X]^{X\setminus Z} = \beta^0(\underline{\sigma}, X)$. 
	Thus the initial case $i = 0$ is clear.

 	By utilizing \cite[Lemma 1.7.2]{FULTON_INTER}, \cite[\S B.5]{FULTON_INTER} and the fact that $\PP(C^\LL\oplus 1)$ is equidimensional, we obtain $c_1\big(\OO_{\PP(C^\LL \oplus 1)}(1)\big) \smallfrown [\PP(C^\LL \oplus 1)] = [\PP(C^\LL)]$.
 	Notice that $\eta : \PP(C) \cong \PP(C^\LL) \rightarrow Z$ is the exceptional divisor of the blow-up $\mathcal{B}=\Bl_Z(X)$;
 	 again, we have the similar identification $\OO_{\PP(C^\LL)}(1) \cong \OO_{\PP(C)}(1) \otimes \eta^*(\LL)$.
 	Thus, for $i \ge 1$, \autoref{thm_van_Gastel} yields  
	\begin{equation}
		\label{eq_vogel_cycle_exc_div}
		\nu^i(\underline{\sigma}, X) \;=\; \eta_*\left(c_1\big(\OO_{\PP(C^\LL)}(1)\big)^{i-1} \smallfrown [\PP(C)]\right).
	\end{equation}
 	As $\PP(C)$ is the exceptional divisor of $\mathcal{B}$, we have 
 	\begin{equation}
 		\label{eq_exc_div_chern}
 		c_1(\OO_\mathcal{B}(-1)) \smallfrown \left[\mathcal{B}\right] = \left[\PP(C)\right].
 	\end{equation}
 	We can now compute 
 	\begin{align*}
		\pi_*\left(c_1\left(\OO_{\mathcal{B}^\LL}(1)\right)^i \smallfrown \left[\mathcal{B}\right]\right) &= \pi_*\left(\big(c_1\left(\OO_{\mathcal{B}}(1)\right) + c_1\left(\pi^*(\LL)\right)\big)c_1\left(\OO_{\mathcal{B}^\LL}(1)\right)^{i-1}\smallfrown \left[\mathcal{B}\right]\right)   \text{ \;\quad (by \autoref{eq_can_line_O(1)})} \\
		&= -\nu^i(\underline{\sigma}, X) + \pi_*\left( c_1\left(\pi^*(\LL)\right)c_1\left(\OO_{\mathcal{B}^\LL}(1)\right)^{i-1}\smallfrown \left[\mathcal{B}\right]\right)  \text{\;\quad (by \autoref{eq_vogel_cycle_exc_div} and \autoref{eq_exc_div_chern})} \\
		&= -\nu^i(\underline{\sigma}, X) + c_1\left(\LL\right)\smallfrown \pi_*\left(c_1\left(\OO_{\mathcal{B}^\LL}(1)\right)^{i-1} \smallfrown \left[\mathcal{B}\right]\right)  \text{\;\; (projection formula)}  \\
		&= -\nu^i(\underline{\sigma}, X) + c_1\left(\LL\right)\smallfrown \beta^{i-1}(\underline{\sigma}, X) \text{\;\quad (by induction).}
 	\end{align*}
 	Finally, by comparing with  \autoref{eq_decomp_cycles}, the equality $\beta^i(\underline{\sigma}, X) \;=\; \pi_*\big(c_1\left(\OO_{\mathcal{B}^\LL}(1)\right)^i \smallfrown \left[\mathcal{B}\right]\big)$ follows
 	(indeed, we have $D_i \cdot \beta^{i-1}(\underline{\sigma}, X)=c_1(\LL) \smallfrown \beta^{i-1}(\underline{\sigma}, X)$).
\end{proof}

\begin{remark}[see \cite{CRPU2}]
Let $R$ be a Noetherian local ring with infinite residue field and $I \subset R$ be an ideal. 
Let $b : B = \Proj(\Rees(I)) \rightarrow \Spec(R)$ be the blow-up of $\Spec(R)$ along $I$.
Let $g_1,\ldots,g_i \in I$ be a sequence of general elements of $I$ and $H_1,\ldots, H_i$ be the corresponding sequence of general hyperplanes on the blow-up $B$.
The \emph{polar cycle} is the cycle associated to the polar scheme $P_i(I, \Spec(R)) := b\left(H_1\cap \cdots \cap H_i\right) \subset \Spec(R)$.
In our global setting, the right-hand side of the blow-up formula in \autoref{prop_blow_up_formula} is obtained, modulo rational equivalence, by intersecting $i$ times with the hyperplane class on the twisted blow-up and then pushing forward to $X$.
On the other hand, the  cycle of $\Spec\left(R/(g_1,\ldots,g_i): I^\infty\right)$ is obtained from the cycle of $\Spec\left(R/(g_1,\ldots,g_i)\right)$ by keeping the components not contained in $V(I)$.
Then \cite[Theorem C(i)]{CRPU2} shows that $P_i(I, \Spec(R))$ equals $\Spec\left(R/(g_1,\ldots,g_i):I^\infty\right)$. 
Thus, modulo rational equivalence, \autoref{prop_blow_up_formula} shows that $\beta^i(\underline{\sigma}, X)$ is the global analogue of the local polar cycle, which justifies the terminology.
\end{remark}

Next, we relate the $\LL$-degrees of the polar cycles $\beta^i(\underline{\sigma}, X)$ to certain Snapper polynomials.
Let 
$$
P_{\LL, Z}(m, n) \;:=\;  \chi\left(\pi^*(\LL)^{\otimes m} \otimes \OO_{\mathcal{B}^\LL}(n)\right) 
$$
be the Snapper polynomial of the line bundles $\pi^*(\LL)$ and $\OO_{\mathcal{B}^\LL}(1)$ on $\mathcal{B} = \Bl_Z(X)$ (see \cite[Appendix B]{KLEIMAN_PICARD}, \cite{KLEIMAN_NUM}, \cite{SNAPPER} for more details on Snapper polynomials).

\begin{corollary}
	\label{cor_numerical_snapper}
	The following statements hold:
	\begin{enumerate}[\rm (i)]
		\item The terms of $P_{\LL, Z}(m, n)$ of degree $d$ are given as follows
		$$
		P_{\LL, Z}(m, n) \;=\; \sum_{i=0}^d \frac{\deg_\LL\left(\beta^i(\underline{\sigma}, X)\right)}{i!(d-i)!} \,m^{d-i} n^{i} \;+\; \text{{\rm(}lower degree terms{\rm)}}.
		$$
		\item For all $m \in \ZZ$ and $n \gg 0$, we have the equality
		$$
		P_{\LL, Z}(m, n) \;=\; \chi\big(\II_Z^n \otimes \LL^{\otimes (m+n)}\big). 
		$$
		\item If $\LL$ is very ample on $X$, we get the equality
		$$
		P_{\LL, Z}(m, n) \;=\; h^0\big(X, \II_Z^n \otimes \LL^{\otimes (m+n)}\big)
		$$
		for all $m \gg 0$ and $n \gg 0$.
	\end{enumerate}
\end{corollary}
\begin{proof}
	(i) By \cite[Theorem B.7, Definition B.8]{KLEIMAN_PICARD}, we have 
	$$
	\chi\left(\pi^*(\LL)^{\otimes m} \otimes \OO_{\mathcal{B}^\LL}(n)\right) \;=\; \sum_{i=0}^d \frac{a_{i}}{i!(d-i)!} \,m^{d-i} n^{i} \;+\; \text{(lower degree terms)}
 	$$
 	where $a_{i}= \int c_1\left(\pi^*(\LL)\right)^{d-i} c_1\left(\OO_{\mathcal{B}^\LL}(1)\right)^i \smallfrown \left[\mathcal{B}\right]$.
 	Then the equalities 
 	$$
 	a_i \;=\; \int c_1\left(\LL\right)^{d-i}  \smallfrown \pi_*\left( c_1\left(\OO_{\mathcal{B}^\LL}(1)\right)^{i} \smallfrown \left[\mathcal{B}\right]\right) \;=\; \deg_\LL\left(\beta^i(\underline{\sigma}, X)\right)
 	$$
 	follow from the projection formula and \autoref{prop_blow_up_formula}.
 	
 	(ii)
 	Let $n \gg 0$ and $m \in \ZZ$.
 	We obtain that $\pi_*\left(\OO_{\mathcal{B}^\LL}(n)\right) \cong \II_Z^n \otimes \LL^{\otimes n}$ and $R^i\pi_*\left(\OO_{\mathcal{B}^\LL}(n)\right)=0$ for $i \ge 1$ (see, e.g., \cite[\S III.8]{HARTSHORNE}).
 	Therefore, the Leray spectral sequence (see, e.g., \cite[\href{https://stacks.math.columbia.edu/tag/01EY}{Tag 01EY}]{stacks-project})
 	$$
 	E_2^{p,q} = \HH^p\left(X,\, \LL^{\otimes m}\otimes R^q\pi_*\left( \OO_{\mathcal{B}^\LL}(n)\right)\right) \;\Longrightarrow\; \HH^{p+q}\left(\mathcal{B},\, \pi^*(\LL)^{\otimes m} \otimes \OO_{\mathcal{B}^\LL}(n)\right)
 	$$
 	degenerates in the second page.
 	Notice that by the projection formula (see, e.g., \cite[Exercise III.8.3]{HARTSHORNE}) we have the isomorphism $\LL^{\otimes m}\otimes R^q\pi_*\left( \OO_{\mathcal{B}^\LL}(n)\right) \cong R^q\pi_*\left( \pi^*(\LL)^{\otimes m} \otimes  \OO_{\mathcal{B}^\LL}(n)\right)$.
 	This yields the isomorphism $\HH^i\left(\cB, \pi^*(\LL)^{\otimes m} \otimes \OO_{\mathcal{B}^\LL}(n)\right) \cong \HH^i\left(X, \II_Z^n \otimes \LL^{\otimes(m+n)}\right)$.
 	Hence the equality
 	$$
 	\chi\left(\pi^*(\LL)^{\otimes m} \otimes \OO_{\mathcal{B}^\LL}(n)\right) \;=\; \chi\left(\II_Z^n \otimes \LL^{\otimes(m+n)}\right)
 	$$	
 	follows.
 	
 	(iii) By assumption, we have a closed immersion $X \hookrightarrow \PP_{\kk}^\ell$ such that $\LL$ is the pullback of $\OO_{\PP_{\kk}^\ell}(1)$.
 	Thus we get a closed immersion 
	\begin{equation}
		\label{eq_very_ample_on_B}
		\iota : \cB^\LL \;\hookrightarrow\;  \PP_{\kk}^\ell \times_\kk \PP_{\kk}^{r-1} \quad\text{ with }\quad \OO_{\cB^\LL}(m,n):=\iota^*\big(\OO_{\PP_{\kk}^\ell \times_\kk \PP_{\kk}^{r-1}}(m,n)\big) \;\cong\; \pi^*(\LL)^{\otimes m} \otimes \OO_{\mathcal{B}^\LL}(n).
	\end{equation}
 	Then, for all $m \gg 0$ and $n \gg 0$ and $i \ge 1$, we obtain 
	\begin{align*}
		\HH^i\big(X, \II_Z^n \otimes \LL^{\otimes(m+n)}\big) &\;\cong\; \HH^i\left(\cB, \pi^*(\LL)^{\otimes m} \otimes \OO_{\mathcal{B}^\LL}(n)\right) \\
		&\;\cong\; \HH^i\left(\cB, \OO_{\cB^\LL}(m,n)\right)\\
		&\;=\; 0
	\end{align*}
 	(see, e.g., \cite[Theorem 1.6]{HYRY_MULTGRAD}).
 	The result now follows from part (ii).
\end{proof}

\begin{remark}
	\label{rem_vanishing}
	Assume that $\LL$ is very ample on $X$. 
	The proof of \autoref{cor_numerical_snapper} yields the vanishing 
	$$
	\HH^i\big(X, \II_Z^n \otimes \LL^{\otimes(m+n)}\big) \;=\;0
	$$
	for all $i \ge 1$, $m \gg 0$ and $n \gg 0$.
\end{remark}

\section{Segre classes and an integral dependence criterion}
\label{sect_main}

In this section we prove \autoref{thmA}, the main result of the paper.
It establishes a criterion for integral dependence formulated through Segre classes.
Our proof builds on van Gastel’s result (\autoref{thm_van_Gastel}) and on the developments made in \autoref{sect_van_Gastel}.
The following setup is used throughout this section.

\begin{setup}
	\label{setup_segre_int_dep}
	Let $\kk$ be a field and $X$ be an equidimensional projective scheme over $\kk$.
	Let $d := \dim(X)$ be the dimension of $X$.
	Let $\LL$ be an ample line bundle on $X$.
	Let $j_W : W \hookrightarrow X$ and $j_Z : Z \hookrightarrow X$ be two closed subschemes of $X$ such that $W \subseteq Z$.
\end{setup}

Let $p > 0$ be a positive integer such that $\LL^{\otimes p}$ is very ample on $X$ and the coherent sheaves $\II_Z \otimes \LL^{\otimes p}$ and $\II_W \otimes \LL^{\otimes p}$ are  generated by global sections.
Let 
$$
\underline{\sigma} = \sigma_1, \ldots, \sigma_r \;\in\; \HH^0\left(X, \II_Z \otimes \LL^{\otimes p}\right) \;\subset\; \HH^0\left(X, \LL^{\otimes p}\right)
$$
and 
$$
\underline{\eta} = \eta_1, \ldots, \eta_s \;\in\; \HH^0\left(X, \II_W \otimes \LL^{\otimes p}\right) \;\subset\; \HH^0\left(X, \LL^{\otimes p}\right)
$$
be sequences of global sections of $\LL^{\otimes p}$ generating the coherent sheaves $\II_Z \otimes \LL^{\otimes p}$ and $\II_W \otimes \LL^{\otimes p}$, respectively.
Notice that $Z = V(\sigma_1) \cap \cdots \cap V(\sigma_r)$ and $W = V(\eta_1) \cap \cdots \cap V(\eta_s)$.
Following \autoref{def_inter_algo} and \autoref{def_vogel_cycle}, we consider the Vogel cycles
$$
\nu(\underline{\sigma}, X) = \sum_{i\ge 0}\nu^i(\underline{\sigma}, X) \in Z_*\left(Z_\llf\right)
\quad \text{ and } \quad 
\nu(\underline{\eta}, X) = \sum_{i\ge 0}\nu^i(\underline{\eta}, X) \in Z_*\left(W_\llf\right).
$$
Due to van Gastel's result (see \autoref{thm_van_Gastel}), we obtain the following equalities 
\begin{align}
	\label{eq_vogel_Z}
	\nu^i(\underline{\sigma}, X)  &\;=\; \sum_{j=0}^i\binom{i-1}{j-1} p^{i-j}c_1(\LL)^{i-j} \smallfrown s^j(Z, X) \\
	\label{eq_vogel_W}
	\nu^i(\underline{\eta}, X)  &\;=\; \sum_{j=0}^i\binom{i-1}{j-1} p^{i-j}c_1(\LL)^{i-j} \smallfrown s^j(W, X) \\ 
	\label{eq_seg_Z} s^i(Z, X) &\;=\; \sum_{j=0}^i \binom{i-1}{j-1} {(-1)}^{i-j} p^{i-j}c_1(\LL)^{i-j} \smallfrown \nu^j(\underline{\sigma}, X) \\
	\label{eq_seg_W} s^i(W, X) &\;=\; \sum_{j=0}^i \binom{i-1}{j-1} {(-1)}^{i-j} p^{i-j}c_1(\LL)^{i-j} \smallfrown \nu^j(\underline{\eta}, X).
\end{align}

The lemma below provides an important equivalence that follows from van Gastel’s result. 

\begin{lemma}
	\label{lem_num_equiv}
	The following conditions are equivalent: 
	\begin{enumerate}[\;\;\rm (a)]
		\item $\deg_\LL\left(s^i(Z, X)\right) = \deg_\LL\left(s^i(W, X)\right)$ for all $i \ge 0$.
		\item $\deg_\LL\left(\nu^i(\underline{\sigma}, X)\right) = \deg_\LL\left(\nu^i(\underline{\eta}, X)\right)$ for all $i \ge 0$.
		\item $\deg_\LL\left(\beta^i(\underline{\sigma}, X)\right) = \deg_\LL\left(\beta^i(\underline{\eta}, X)\right)$ for all $i \ge 0$.
	\end{enumerate}  
\end{lemma}
\begin{proof}
	(a) $\Rightarrow$ (b):		
	By utilizing \autoref{eq_vogel_Z}, we obtain
	\begin{align*}
		\deg_\LL\left(\nu^i(\underline{\sigma}, X)\right) &\;=\; \int c_1(\LL)^{d-i} \smallfrown \nu^i(\underline{\sigma}, X) \\
		&\;=\;   \sum_{j=0}^i\binom{i-1}{j-1} p^{i-j} \int c_1(\LL)^{d-i} c_1(\LL)^{i-j} \smallfrown s^j(Z, X) \\ 
		&\;=\;   \sum_{j=0}^i\binom{i-1}{j-1} p^{i-j} \deg_\LL\left(s^j(Z, X)\right).
	\end{align*}
	Similarly, \autoref{eq_vogel_W} yields 
	$$
	\deg_\LL\left(\nu^i(\underline{\eta}, X)\right) \;=\;  \sum_{j=0}^i\binom{i-1}{j-1} p^{i-j} \deg_\LL\left(s^j(W, X)\right).
	$$ 
	So the implication (a) $\Rightarrow$ (b) follows.
	
	(b) $\Rightarrow$ (a): This implication can be obtained similarly from \autoref{eq_seg_Z} and \autoref{eq_seg_W}.
	
	(b) $\Leftrightarrow$ (c): 
	As in \autoref{def_inter_algo}, let $D_i$ be the set of zeros of a generic linear combination of the sections $\underline{\sigma} = \sigma_1,\ldots,\sigma_r$.
	We have $D_i \cdot \beta^{i-1}(\underline{\sigma},X) = c_1(\LL^{\otimes p}) \smallfrown \beta^{i-1}(\underline{\sigma},X)$ for all $i \ge 1$ (see \autoref{def_inter_algo} and \autoref{rem_base_change}).
	So, we get 
	$$
	\deg_{\LL^{\otimes p}}(X) = \deg_{\LL^{\otimes p}}\left(\nu^0(\underline{\sigma}, X)\right) + \deg_{\LL^{\otimes p}}\left(\beta^0(\underline{\sigma}, X)\right)
	$$ 
	and 
	$$
	\deg_{\LL^{\otimes p}}\left(\beta^{i-1}(\underline{\sigma}, X)\right) \;=\; \deg_{\LL^{\otimes p}}\left(\nu^{i}(\underline{\sigma}, X)\right) + \deg_{\LL^{\otimes p}}\left(\beta^{i}(\underline{\sigma}, X)\right) \quad \text{ for all \quad $i \ge 1$.}
	$$
	By the same token, we obtain 
	$$
	\deg_{\LL^{\otimes p}}(X) = \deg_{\LL^{\otimes p}}\left(\nu^0(\underline{\eta}, X)\right) + \deg_{\LL^{\otimes p}}\left(\beta^0(\underline{\eta}, X)\right)
	$$ 
	and
	$$
	\deg_{\LL^{\otimes p}}\left(\beta^{i-1}(\underline{\eta}, X)\right) \;=\; \deg_{\LL^{\otimes p}}\left(\nu^{i}(\underline{\eta}, X)\right) + \deg_{\LL^{\otimes p}}\left(\beta^{i}(\underline{\eta}, X)\right) \quad \text{ for all \quad $i \ge 1$.}
	$$
	Thus the claimed equivalence follows. 
\end{proof}

We are now ready to prove our main result.

\begin{theorem}
	\label{thm_main}
	Assume \autoref{setup_segre_int_dep}.
	Then the following six conditions are equivalent:
	\begin{enumerate}[\;\;\rm (a)]
		\item $\II_W$ is integral over $\II_Z$. 
		\item[\;\;\rm(b$'$)] $s(Z, X) = {j}_*\left(s(W, X)\right)$ in $A_*(Z)$, where $j : W \hookrightarrow Z$ is the natural closed immersion.
		\item ${j_Z}_*\left(s(Z, X)\right) = {j_W}_*\left(s(W, X)\right)$ in $A_*(X)$.
		\item $\deg_\LL\left(s^i(Z, X)\right) = \deg_\LL\left(s^i(W, X)\right)$ for all $i \ge 0$.
		\item $\deg_\LL\left(\nu^i(\underline{\sigma}, X)\right) = \deg_\LL\left(\nu^i(\underline{\eta}, X)\right)$ for all $i \ge 0$.
		\item $\deg_\LL\left(\beta^i(\underline{\sigma}, X)\right) = \deg_\LL\left(\beta^i(\underline{\eta}, X)\right)$ for all $i \ge 0$.
	\end{enumerate} 
\end{theorem}
\begin{proof}
	(a) $\Rightarrow$ (b$'$): This implication is a consequence of the birational invariance of Segre classes. 	
	For completeness, we include a short argument.
	Assume that $\overline{\II_Z} = \overline{\II_W}$.
	Thus $Z$ and $W$ have the same support, and so we may identify their Chow groups.
	Let $X_1,\ldots,X_k$ be the irreducible components of $X$ and $m_i$ be the geometric multiplicity of $X_i$ in $X$.
	Let $Z_i = Z \cap X_i$ and $W_i = W \cap X_i$.
	By \cite[Lemma 4.2]{FULTON_INTER}, we have $s(Z, X) = \sum_{i=1}^k m_i s(Z_i, X_i) \in A_*(Z)$ and $s(W, X) = \sum_{i=1}^k m_i s(W_i, X_i) \in A_*(W)$.
	Thus to show the equality $s(Z, X) = s(W, X) \in A_*(Z) = A_*(W)$, we may assume that $X$ is irreducible.
	
	If we had the equality $W_{\rm red} = X_{\rm red}$ (equivalently, the equality $Z_{\rm red} = X_{\rm red}$), we would obtain 
	$$
	s(Z, X) \;=\; [X] \;=\; s(W, X);
	$$
	this follows for instance by using \autoref{thm_van_Gastel} and \autoref{def_inter_algo}.
	Therefore, we may assume that $W$ and $Z$ are nowhere dense on $X$.
	
	Let $b_Z : \Bl_Z(X) \rightarrow X$ and $b_W : \Bl_W(X) \rightarrow X$ be the blow-ups of $X$ along $Z$ and $W$, respectively. 
	Let $E_ZX$ and $E_WX$ be the corresponding exceptional divisors.
	Consider the natural morphisms $\psi : E_WX \rightarrow E_ZX$, $\eta : E_ZX \rightarrow Z$ and $\overline{\eta} : E_WX  \rightarrow W$.
	Under the  finite birational morphism $\Bl_W(X) \rightarrow \Bl_Z(X)$ the pullback of $E_ZX$ is equal to $E_WX$; indeed, recall that $\II_Z{\II_W}^k = {\II_W}^{k+1}$ for $k \gg 0$.  
	Then, by applying \cite[Proposition 4.2(a)]{FULTON_INTER} successively, we obtain
	\begin{align*}
		s\left(Z, X\right) &\;=\; \eta_*\big(s\left(E_ZX,\, \Bl_Z(X)\right)\big)\\
		&\;=\; \eta_*\left(\psi_*\big(s\left(E_WX,\, \Bl_W(X)\right)\big)\right) \\
		&\;=\; \overline{\eta}_*\big(s\left(E_WX,\, \Bl_W(X)\right)\big)\\
		&\;=\; s\left(W, X\right) \;\;\in\;\; A_*(Z) = A_*(W).
	\end{align*}
	This concludes the proof of this part.

	\smallskip
	
	(b$'$) $\Rightarrow$ (b) $\Rightarrow$ (c): These implications are clear by definition. 

	\smallskip
	
	(c) $\Leftrightarrow$ (d) $\Leftrightarrow$ (e): These equivalences were proved in \autoref{lem_num_equiv}.
	
	\smallskip
	
	(e) $\Rightarrow$ (a): 
	Assume that $\deg_\LL\left(\beta^i(\underline{\sigma}, X)\right) = \deg_\LL\left(\beta^i(\underline{\eta}, X)\right)$ for all $i \ge 0$.
	Consider the short exact sequence 
	$$
	0 \;\rightarrow\; \II_Z^{n} \otimes \LL^{\otimes 2pn} \;\rightarrow\; \II_W^{n} \otimes \LL^{\otimes 2pn} \;\rightarrow\; \II_W^{n}/\II_Z^{n} \otimes \LL^{\otimes 2pn} 
	\;\rightarrow\; 0.
	$$
	By \autoref{rem_vanishing}, we have $\HH^i\big(X, \II_Z^{n} \otimes \LL^{\otimes 2pn}\big) = 0$ for all $i \ge 1$ and $n \gg 0$.
	Thus the corresponding long exact sequence in cohomology and \autoref{cor_numerical_snapper} yield 
	\begin{align*}
		\lim_{n \to \infty} \frac{h^0\left(X, \II_W^{n}/\II_Z^{n} \otimes \LL^{\otimes 2pn} \right)}{n^d} &\;=\; \lim_{n \to \infty} \frac{h^0\left(X, \II_W^{n} \otimes \LL^{\otimes 2pn} \right)-h^0\left(X, \II_Z^{n} \otimes \LL^{\otimes 2pn} \right)}{n^d} \\
		&\;=\; \lim_{n \to \infty} \frac{P_{\LL^{\otimes p}, W}\left(n, n\right) - P_{\LL^{\otimes p}, Z}\left(n, n\right)}{n^d} \\
		&\;=\; \lim_{n \to \infty} \frac{\sum_{i=0}^d \left(\deg_\LL\left(\beta^i(\underline{\eta},X)\right)-\deg_\LL(\beta^i\left(\underline{\sigma},X)\right)\right)\frac{p^{d-i}n^d}{i!(d-i)!}}{n^d} \\
		&\;=\; 0.
	\end{align*}
	Finally, by \autoref{prop_hard} below, $\Rees(\II_W)$ must be of finite type as a module over $\Rees(\II_Z)$.
	This concludes the proof of the theorem.
\end{proof}

The following proposition was our main tool in the proof of \autoref{thm_main}.
Although the proof is technical, the idea is that a failure of integral dependence leaves a detectable trace on the blow-up of $X$ along $W$. 
If $\II_W$ were not integral over $\II_Z$, then $\fProj_X\left(\Rees(\II_W)/\Rees(\II_Z)_+\Rees(\II_W)\right)$ would be nonempty, and we could choose a closed point $y$ on this closed subscheme of $\Bl_W(X)$ (see \hyperref[chunk_int_dep]{\S 2.8}).
The Hilbert-Samuel polynomial of the local ring $\OO_{\Bl_W(X),y}$ then gives, after suitable ample twists, a lower bound for the asymptotic growth of $h^0\left(X, \II_W^{n}/\II_Z^{n} \otimes \LL^{\otimes 2pn} \right)$.
This local-to-global argument relies in a fundamental way on the work of Kleiman and Thorup on Buchsbaum-Rim polynomials (see \cite[\S 8]{KLEIMAN_THORUP_MIXED}).
Finally, since $\Bl_W(X)$ would be equidimensional of dimension $d$, the local ring of $y$ would have dimension $d$, and so the corresponding Hilbert-Samuel polynomial would have degree $d$.

\begin{proposition}
	\label{prop_hard}
	Assume that $\Rees(\II_W)$ is not of finite type as a module over $\Rees(\II_Z)$.
	 Then we have the following strict inequality 
	 $$
	 \lim_{n \to \infty} \frac{h^0\left(X, \II_W^{n}/\II_Z^{n} \otimes \LL^{\otimes 2pn} \right)}{n^d} \;>\; 0.
	 $$
\end{proposition}
\begin{proof}
	To simplify notation, let 
	$$
	\sC \;:=\; \Rees^{\LL^{\otimes p}}(\II_Z) \;=\; \bigoplus_{n\ge 0} \II_Z^n \otimes \LL^{\otimes pn} \quad \text{ and } \quad  \sD \;:=\; \Rees^{\LL^{\otimes p}}(\II_W) \;=\; \bigoplus_{n\ge 0} \II_W^n \otimes \LL^{\otimes pn}.
	$$
	Let $\pi : \sY = \fProj_X\left(\sD\right) \rightarrow X$ be the natural projection.
	We cannot have that $W_{\rm red} = X_{\rm red}$, because otherwise the ideal sheaves $\II_Z \subseteq \II_W$ would be nilpotent, and so they would have the same integral closure (equal to the nilradical ideal sheaf of $X$). 
	It then follows that $\sY$ is equidimensional of dimension $d$ (see \autoref{rem_blowup_irr_comp}). 
	By assumption, we have that $\fProj_X\left(\sD/\sC_+\sD\right) \neq \varnothing$, where  $\sC_+\sD \subset \sD$ denotes the graded ideal in $\sD$ generated by $\sC_+ = \bigoplus_{n\ge 1} \sC_n$ (see \hyperref[chunk_int_dep]{\S 2.8}).
	Let $y \in \sY$ be a closed point lying in  $\fProj_X\left(\sD/\sC_+\sD\right)$.
	Let $\II_y \subset \OO_\sY$ be the ideal sheaf of $y$. 
	Let $\aaa \subset \sD$ be the largest graded ideal in $\sD$ whose sheafification is $\II_y = \widetilde{\aaa}$.
	Since $y$ lies in $\fProj_X(\sD/\sC_+\sD)$, the maximality of $\aaa$ gives $\aaa \supset \sC_+\sD$.
	As $\pi$ is a closed morphism, the closed point $y \in \sY$ is mapped to a closed point $x  = \pi(y)$ in $X$.
	
	The function $n \mapsto \dim_{\kk}\left(\OO_{\sY, y} / \mm_y^{n} \right) = \left[\kappa(y):\kk\right] \cdot \length_{\OO_{\sY, y}}\left(\OO_{\sY, y} / \mm_y^{n}\right)$ eventually (for $n \gg 0$) becomes a polynomial $P_y(n) \in \QQ[n]$ of degree equal to $\dim(\OO_{\sY, y})$ (see, e.g., \cite[Chapter 5]{MATSUMURA}).
	We have that $\dim(\OO_{\sY, y}) =  d$ because $\sY$ is equidimensional of dimension $d$ and $y \in \sY$ is a closed point. 
	
	For organizational purposes, we encapsulate the most technical arguments in two claims that we prove below.
	
	\begin{claim}
		\label{claim1}
		Let $\delta \ge 0$ be large enough.
		Consider the function 
		$$
		\alpha(v, n, m) \;:=\; h^0\big(X, \sD_{v+n\delta} / \left[\aaa^{n}\right]_{v+n\delta} \otimes \LL^{\otimes m}\big)
		$$ 
		Then, for $v \gg 0$, $n \gg 0$ and $m \in \ZZ$, we have the equality
		$$
		\alpha(v, n, m) \;=\; P_y(n).
		$$
	\end{claim}	
	\begin{proof}[Proof of the claim]
		Let $\sD_x = \bigoplus_{n \ge 0} \left(\sD_n\right)_x$ and $\aaa_x = \bigoplus_{n \ge 0} \left(\aaa_n\right)_x$ be the stalks at $x$ of the quasi-coherent $\OO_X$-modules $\sD$ and $\aaa$, respectively.
		We can see $\sD_x$ as a standard graded algebra over the Noetherian local ring $\OO_{X, x}$ and $\aaa_x \subset \sD_x$ as a graded ideal in $\sD_x$.
		Since $\delta \ge 0$ is assumed to be large enough, there is a finite set of homogeneous elements of degree at most $\delta$ that generates the ideal  $\aaa_x \subset \sD_x$.
		Let $H := \left[\aaa_x\right]_\delta \subset \left[\sD_x\right]_\delta$ be the graded part of $\aaa_x$ of degree $\delta$.
		Then, by  \cite[Proposition 8.2]{KLEIMAN_THORUP_MIXED}, the function 
		$$
		\lambda(v,n) \;:=\; \dim_\kk\left(\left[\sD_x\right]_{v+n\delta} /  H^n\left[\sD_x\right]_v\right) \;=\; \left[\kappa(x) : \kk\right] \cdot  \length_{\OO_{X, x}}\left(\left[\sD_x\right]_{v+n\delta} /  H^n\left[\sD_x\right]_v\right)
		$$
		eventually (for $v \gg 0$ and $n \gg 0$) becomes a polynomial $P(v, n) \in \QQ[v,n]$ of total degree at most $d=\dim(\sY)$.
		
		Next, we show that $\alpha(v, n, m)$ equals the function $\lambda(v,n)$.
		Let $n \ge 1$.		 
		Since $\sD/\aaa^{n}$ (as a quasi-coherent $\OO_X$-module) is only supported at the point $x$, we can see $\sD/\aaa^{n}$  as a standard graded algebra over an Artinian local ring which is a quotient of $\OO_{X, x}$.
		Consequently,  the isomorphisms
		\begin{align}
			\label{eq_alpha_lambda}
			\begin{split}
				\HH^0\big(X, \sD_{v+n\delta} / \left[\aaa^{n}\right]_{v+n\delta} \otimes \LL^{\otimes m}\big)  &\;\cong\; \left(\sD_{v+n\delta} / \left[\aaa^{n}\right]_{v+n\delta} \otimes \LL^{\otimes m}\right)_x\\
				&\;\cong\; \left[\sD_x\right]_{v+n\delta} /  H^{n}\left[\sD_x\right]_v
			\end{split}			
		\end{align}
		follow.
		Thus $\alpha(v,n,m) = \lambda(v, n)$ for all $v \ge 0$, $n \ge 1$ and $m \in \ZZ$.
		
		The final step is to show that the two polynomials $P(v, n)$ and $P_y(n)$ actually coincide.	
		Since $\OO_\sY/ \II_y^{n}$ is only supported at the point $y$, we get 
		\begin{equation}
			\label{eq_local_ring_y}
			\HH^0\left(\sY, \OO_{\sY}/\II_y^{n} \otimes \OO_\sY(v+n\delta) \otimes \pi^*(\LL)^{\otimes m}\right) \;\cong\; \OO_{\sY, y} / \mm_y^{n} 
		\end{equation}
		where $\OO_{\sY, y}$ is the local ring of $y$ and $\mm_y \subset \OO_{\sY, y}$ is the corresponding maximal ideal.
		We fix $n_0$ large enough so that $\lambda(v, n_0)$
		eventually becomes the polynomial $P(v, n_0) \in \QQ[v]$ in $v$.
		For $v \gg 0$, since 
		$$
		\pi_*\left(\OO_{\sY}/\II_y^{n_0} \otimes \OO_\sY(v+n_0\delta) \otimes \pi^*(\LL)^{\otimes m}\right) \;\cong\; \sD_{v+n_0\delta} / \left[\aaa^{n_0}\right]_{v+n_0\delta} \otimes \LL^{\otimes m},
		$$
		we also obtain
		\begin{equation}
			\label{eq_pow_P_y}
			\HH^0\left(\sY, \OO_{\sY}/\II_y^{n_0} \otimes \OO_\sY(v+n_0\delta) \otimes \pi^*(\LL)^{\otimes m}\right) \;\cong\; \HH^0\big(X, \sD_{v+n_0\delta} / \left[\aaa^{n_0}\right]_{v+n_0\delta} \otimes \LL^{\otimes m}\big). 
		\end{equation}
		By combining \autoref{eq_alpha_lambda}, \autoref{eq_local_ring_y} and \autoref{eq_pow_P_y}, we get
		$$
		\alpha(v,n_0,m) \;=\; \lambda(v, n_0) \;=\; \dim_{\kk}\left(\OO_{\sY, y} / \mm_y^{n_0} \right)
		$$
		for $v \gg 0$ and $m \in \ZZ$.
		Hence $P(v, n_0) \in \QQ[v]$ is the constant polynomial equal to $\dim_{\kk}\left(\OO_{\sY, y} / \mm_y^{n_0} \right)$.
		Since the last equality holds for any $n_0$ large enough, the following equality of polynomials
		$$
		P(v, n) \;=\; P_y(n)
		$$
		follows. 
		This concludes the proof of the claim.
	\end{proof}

	\begin{claim}
		\label{claim2}
		Let $\delta \ge 0$ be large enough.
		Then 
		$$
		\HH^i\left(X, \left[\aaa^n\right]_{v+n\delta} \otimes \LL^{\otimes p(m+n\delta)}\right) \;=\; 0
		$$
		for all $i \ge 1$, $m \gg 0$, $v \gg 0$ and $n \gg 0$.
	\end{claim}
	\begin{proof}[Proof of the claim]
		We have a closed immersion $X \hookrightarrow \PP_{\kk}^\ell$ such that $\LL^{\otimes p}$ is the pullback of $\OO_{\PP_{\kk}^\ell}(1)$.
		As in \autoref{eq_very_ample_on_B}, the line bundle $\sM := \pi^*(\LL^{\otimes p}) \otimes \OO_{\sY}(1)$ is very ample on $\sY = \Bl_W^{\LL^{\otimes p}}(X)$.
		Thus we can choose $\delta \ge 0$ large enough so that $\II_y \otimes \sM^{\otimes \delta}$ is generated by global sections. 
		We consider the twisted blow-up 
		$$
		\sW  := \Bl_y^{\sM^{\otimes \delta}}(\sY) \;=\; \fProj_\sY\left(\Rees^{\sM^{\otimes \delta}}(\II_y)\right) \;=\; \fProj_\sY\left(\bigoplus_{n \ge 0}\II_y^n \otimes \sM^{\otimes n\delta}\right).
		$$
		Let $\omega : \sW \rightarrow X$ and $\gamma : \sW \rightarrow \sY$ be the natural projections.
		Similarly to \autoref{rem_nice_immersion}, by choosing global sections $\beta_1, \ldots,\beta_h$ generating $\II_y \otimes \sM^{\otimes \delta}$, we get a closed immersion $\sW \hookrightarrow \sY \times_\kk \PP_{\kk}^{h-1}$ such that $\OO_\sW(1)$ is the pullback of $\OO_{\PP_{\kk}^{h-1}}(1)$ via the natural projection $\sW \rightarrow \PP_{\kk}^{h-1}$.
		By combining with \autoref{rem_nice_immersion} and \autoref{eq_very_ample_on_B}, we obtain a closed immersion 
		$$
		\iota : \sW \;\hookrightarrow\; \PP_{\kk}^\ell \times_\kk \PP_{\kk}^{s-1} \times_\kk \PP_{\kk}^{h-1}
		$$
		such that $\omega^*(\LL)^{\otimes mp} \otimes \gamma^*\left(\OO_\sY(v)\right) \otimes \OO_\sW(n)$ is isomorphic to the pullback 
		$$
		\sW(m,v,n)\;:=\;\iota^*\big(\OO_{\PP_{\kk}^\ell \times_\kk \PP_{\kk}^{s-1} \times_\kk \PP_{\kk}^{h-1}}(m,v, n)\big).
		$$ 
		The same arguments as in the proof of \autoref{cor_numerical_snapper} yield the vanishing
		\begin{align*}
			\HH^i\left(X, \left[\aaa^n\right]_{v+n\delta} \otimes \LL^{\otimes p(m+n\delta)}\right) &\;\cong\; \HH^i\left(\sW, \omega^*(\LL)^{\otimes mp} \otimes \gamma^*\left(\OO_\sY(v)\right) \otimes \OO_\sW(n)\right) \\
			&\;\cong\; \HH^i\left(\sW, \OO_{\sW}(m,v,n)\right)\\ &\;=\; 0
		\end{align*}
		for all $i \ge 1$, $m \gg 0$, $v \gg 0$ and $n \gg 0$ (see, e.g., \cite[Theorem 1.6]{HYRY_MULTGRAD}).
	\end{proof}
	
	Fix $\delta \ge 0$, $m = v \ge 0$ and $n \ge 0$ large enough so that both \autoref{claim1} and \autoref{claim2} hold. 
	By \autoref{rem_vanishing}, we may also assume that $\HH^i\left(X, \sC_{v + n\delta} \otimes \LL^{\otimes p(v+n\delta)}\right)=0$ for all $i \ge 1$.
	Since we have the inclusions $\sC_{v + n\delta} \subseteq \left[\left(\sC_+\sD\right)^{n}\right]_{v+n\delta} \subseteq \left[\aaa^n\right]_{v+n\delta}$, we obtain the short exact sequence 
	$$
	0 \;\rightarrow\; \sC_{v + n\delta} \;\rightarrow\; \left[\aaa^n\right]_{v+n\delta} \;\rightarrow\; \left[\aaa^n\right]_{v+n\delta}/\sC_{v + n\delta}  \;\rightarrow\; 0.
	$$
	By twisting it with $\LL^{\otimes p(v+n\delta)}$ and considering the corresponding long exact sequence in cohomology, the vanishings of \autoref{rem_vanishing} and \autoref{claim2} yield
	\begin{equation}
		\label{eq_vanish_quot}
		\HH^i\left(X, \left[\aaa^n\right]_{v+n\delta}/\sC_{v + n\delta} \otimes \LL^{\otimes p(v+n\delta)}\right) \;=\; 0
	\end{equation}
	for all $i \ge 1$.
	
	We also have the following short exact sequence
	$$
	0 \;\rightarrow\; \left[\aaa^n\right]_{v+n\delta}/\sC_{v + n\delta}  \;\rightarrow\; \sD_{v + n\delta} / \sC_{v + n\delta} \;\rightarrow\;  \sD_{v + n\delta} / \left[\aaa^n\right]_{v+n\delta}   \;\rightarrow\; 0.
	$$
	By twisting it with $\LL^{\otimes p(v+n\delta)}$, considering the corresponding long exact sequence in cohomology, and utilizing \autoref{eq_vanish_quot} and \autoref{claim1}, we obtain the inequalities 
	\begin{align*}
		\lim_{n \to \infty} \frac{h^0\left(X, \II_W^{v+n\delta} / \II_Z^{v+n\delta} \otimes \LL^{\otimes 2p(v+n\delta)} \right)}{\left(v+n\delta\right)^d} 
		&\;=\; \lim_{n \to \infty} \frac{h^0\left(X, \sD_{v + n\delta}/\sC_{v + n\delta} \otimes \LL^{\otimes p(v+n\delta)} \right)}{\left(v+n\delta\right)^d} 
		\\
		&\;\ge\; 
		\lim_{n \to \infty} \frac{h^0\left(X, \sD_{v + n\delta} / \left[\aaa^n\right]_{v+n\delta} \otimes \LL^{\otimes p(v+n\delta)} \right)}{\left(v+n\delta\right)^d} \\
		&\;=\; 	\lim_{n \to \infty} \frac{P_y(n)}{\left(v+n\delta\right)^d} \\
		&\;>\; 0.
	\end{align*}
	Finally, the proof of the proposition is now complete.
\end{proof}

The following example shows that \autoref{thm_main} is sharp: the assumption that $\LL$ be ample appears to be essential, as big and nef line bundles do not suffice.

\begin{example}[Line bundles that are big and nef do not suffice in \autoref{thm_main}]
	\label{examp}
	Assume $\kk$ is algebraically closed.
	Consider the following two successive blow-ups: 
	\begin{enumerate}[\rm (i)]
		\item Let $\pi_1 : Y = \Bl_{p}\left(\PP_\kk^3\right) \rightarrow \PP_\kk^3$ be the blow-up of $\PP_{\kk}^3$ along a closed point $p \in \PP_\kk^3$.
		Notice that the exceptional divisor $E_p\PP_\kk^3$ of $Y=\Bl_{p}\left(\PP_\kk^3\right)$ is isomorphic to $\PP_{\kk}^2$.
		\item Choose a line $C$ in $\PP_\kk^2 \cong E_p\PP_\kk^3$.
		Let $\pi_2 : X = \Bl_C(Y) \rightarrow Y$ be the blow-up of $Y$ along $C$.
		Notice that $X$ is a smooth threefold. 
		Let $E = E_CY$ be the exceptional divisor of $X = \Bl_C(Y)$.
	\end{enumerate}
	Let $\pi = \pi_1 \circ \pi_2 : X \rightarrow \PP_\kk^3$ be the natural projection.
	Let $\mathcal{L} = \pi^*\big(\OO_{\PP_\kk^3}(1)\big)$ be the line bundle on $X$  given as the pullback of $\OO_{\PP_\kk^3}(1)$.
	Notice that $\mathcal{L}$ is big and nef.
	Let $k \ge 2$ and consider the effective Cartier divisor $kE$.
	As closed subschemes,  we have $E \subsetneq kE \subset X$ (although they have the same support).
	We also obtain that their Segre classes are different
	$$
	s(E, X) \;=\; E - E^2 + E^3  \;\neq\; kE -k^2E^2 + k^3E^3  \;=\; s(kE, X);
	$$	
	here $E^i$ denotes the self-intersection of $E$
	(see \cite[Proposition 4.1(a)]{FULTON_INTER}).
	However, we can compute the vanishings
	$$
	\deg_{\mathcal{L}}\left(E\right) \;=\; \deg_{\mathcal{L}}\left(E^2\right) \;=\; \deg_{\mathcal{L}}\left(E^3\right) \;=\; 0.
	$$
	This shows that the big and nef line bundle $\mathcal{L}$ cannot be used in \autoref{thm_main}.
\end{example}
\begin{proof}
	We only need to show the claimed vanishings. 
	The projection formula yields 
	$$
	\deg_\mathcal{L}(E) \;=\; \int c_1\big(\OO_{\PP_\kk^3}(1)\big)^2 \smallfrown \pi_*\left([E]\right) \quad \text{ and } \quad \deg_\mathcal{L}(E^2) \;=\; \int c_1\big(\OO_{\PP_\kk^3}(1)\big) \smallfrown \pi_*\left(E \cdot [E]\right).
	$$
	By construction, $E$ is mapped onto the point $p$ via $\pi$.
	Since $[E] \in A_2(X)$ and $E\cdot [E] \in A_1(X)$, by dimension reasons and the definition of proper pushforward, we get $\pi_*([E]) = 0$ and $\pi_*(E \cdot [E]) = 0$ (see \cite[\S 1.4]{FULTON_INTER}).
	It then follows that $\deg_{\mathcal{L}}\left(E\right) = \deg_{\mathcal{L}}\left(E^2\right) = 0$.
	
	It remains to show that $\deg_\mathcal{L}(E^3)=\int E^3 \cdot [X] = 0$.
	Let $N = N_CY$ be the normal bundle to $C$ in $Y$.
	We have the isomorphism $E \cong \PP(N)$ (see \cite[B.7.1]{FULTON_INTER}).
	Let $\eta : E \cong \PP(N) \rightarrow C$ be the natural projection and $\xi = c_1(\OO_{\PP(N)}(1))$.
	From \cite[Example 8.3.4]{FULTON_INTER}, we have 
	$$
	A^*\left(\PP(N)\right) \;\cong\; A^*(C)[\xi]/\left(\xi^2+c_1(N)\xi\right);
	$$ notice that $c_2(N)=0$ since $\dim(C)=1$.
	Then we can compute 
	\begin{align*}
	\int E^3 \cdot [X] &\;=\; \int E^2 \cdot [E] \\
	&\;=\; \int \left(-c_1\left(\OO_{\PP(N)}(1)\right)\right)^2 \smallfrown \left[\PP(N)\right]  \text{\;\;\;(since $\OO_X(E) \cong \OO_X(-1)$)}\\
	&\;=\; -\int c_1(N) c_1\left(\OO_{\PP(N)}(1)\right) \smallfrown \left[\PP(N)\right] \text{\;\;\;(since $\xi^2+c_1(N)\xi=0$ in $A^*(\PP(N))$)}\\
	&\;=\; -\int c_1(N) \smallfrown \eta_*\left( c_1\left(\OO_{\PP(N)}(1)\right) \smallfrown \left[\PP(N)\right]\right) \text{\;\;\;(projection formula)}\\
	&\;=\; -\int c_1(N) \smallfrown [C] \text{\;\;\;(by \cite[Proposition 3.1(a)(ii)]{FULTON_INTER}).}
	\end{align*}
	Let $D = E_p\PP_\kk^3$ be the exceptional divisor of $Y=\Bl_{p}\left(\PP_\kk^3\right)$.
	Since $\iota : C \hookrightarrow D$ and $j:D \hookrightarrow Y$ are regular embeddings, \cite[B.7.4]{FULTON_INTER} yields the short exact sequence of vector bundles 
	$$
	0 \;\rightarrow\; N_CD \;\rightarrow\; N_CY \;\rightarrow\; \iota^*\left(N_DY\right) \;\rightarrow\; 0
	$$
	on $C$ (see also \cite[\href{https://stacks.math.columbia.edu/tag/063N}{Tag 063N}]{stacks-project}).
	By exploiting the isomorphisms $D \cong \PP_\kk^2$ and $C \cong \PP_\kk^1$, we obtain $N_CD \cong \OO_{\PP_\kk^1}(1)$ and $\iota^*(N_DY) \cong \OO_{\PP_\kk^1}(-1)$.
	Therefore, the Whitney sum formula yields
	$$
	\int E^3 \cdot [X] \;=\; -\int c_1(N) \smallfrown [C] \;=\; -(1-1) \;=\; 0.
	$$
	This concludes the proof.
\end{proof}

\section{Segre zeta functions and integral dependence}
\label{sect_zeta}

In this section, we show that Aluffi's Segre zeta function \cite{aluffi2017segre} yields an integral dependence criterion for homogeneous ideals in a polynomial ring. 
This power series encodes information about the Segre classes that arise when the ideal is extended to projective spaces of arbitrarily large dimension. 
Further developments regarding Segre zeta functions were made in \cite{JORGENSON, aluffi2024lorentzian,  MIXED_SEG_ZETA}.

Let $R = \kk[x_0,\ldots,x_n]$ be a polynomial ring over a field $\kk$ and $I \subset R$ be a homogeneous ideal of $R$.
Let $f_{1}, f_{2}, \dotsc, f_{r} \in R$ be homogeneous generators of $I$, with the degree of $f_{i}$ equal to $d_{i} = \deg\left(f_{i}\right)$.
Denote by $\iota : Z \hookrightarrow \PP_\kk^n$ the closed subscheme defined by $I \subset R$ in $\PP_\kk^n = \Proj(R)$. 
For each integer $N \ge n$, we consider the following objects:
\begin{itemize}[\;\;--]
	\item $R^N = \kk[x_0,\ldots,x_n, x_{n+1},\ldots, x_N]$ a polynomial ring containing $R$ and $\PP_\kk^N = \Proj(R^N)$.
	\item $\iota_N : Z^N \hookrightarrow \PP_\kk^N$ the closed subscheme defined by the extension of $I$ to $R^N$.
\end{itemize}

Aluffi \cite{aluffi2017segre} introduced and studied the \emph{Segre zeta function} $\zeta_I(t) = \sum_{k = 0}^\infty a_kt^k \in \ZZ[\![t]\!]$ of the homogeneous ideal $I \subset R$. 
This power series is characterized by the fact that, for all $N \ge n$, the class 
$$
\big(a_0 + a_1H + \cdots + a_NH^N\big) \smallfrown \left[\PP_\kk^N\right] \;\in\; A^*\left(\PP_\kk^N\right)
$$
equals the pushforward ${\iota_N}_*\left(s(Z^N, \PP_\kk^N)\right)$ to $\PP_\kk^N$ of the Segre class $s(Z^N, \PP_\kk^N)$; here $H$ denotes the class of a hyperplane in $\PP_\kk^N$.
An important result concerning Segre zeta functions is their \emph{rationality}.

\begin{theorem}[{Aluffi \cite{aluffi2017segre}}]
	We can write 
	$$
	\zeta_I(t) \;=\; \frac{P(t)}{(1+d_1t)(1+d_2t)\cdots(1+d_rt)}	
	$$
	where $P(t) \in \NN[t]$ is a polynomial with nonnegative integer coefficients.
\end{theorem}

By utilizing the birational invariance of Segre classes, Aluffi \cite{aluffi2017segre} showed that the Segre zeta function only depends on the integral closure of the ideal.
As a consequence of \autoref{thmA}, we obtain that, in fact, the Segre zeta function yields an integral dependence criterion.

\begin{corollary}
	\label{thm_main_zeta}
	Let $J \subset R = \kk[x_0,\ldots,x_n]$ be a homogeneous ideal such that $J \supseteq I$.
	Then the following two conditions are equivalent:
	\begin{enumerate}[\;\;\rm (a)]
		\item $J$ is integral over $I$.
		\item $\zeta_I(t) = \zeta_J(t)$.
	\end{enumerate} 
\end{corollary}
\begin{proof}
	(a) $\Rightarrow$ (b): This implication is a direct consequence of the birational invariance of Segre classes (see \cite[Proposition 5.12]{aluffi2017segre} or the proof of the implication (a) $\Rightarrow$ (b) in \autoref{thm_main}).
	
	\smallskip
	
	(b) $\Rightarrow$ (a):
	Assume that $\zeta_I(t) = \zeta_J(t)$.
	Write $J = (g_1,\ldots,g_s)$,  where the $g_i$'s are homogeneous generators. 
	With a similar notation as before, set $W = V(J) \subset \PP_\kk^n$ and $W^N = V(JR^N) \subset \PP_\kk^N$ for all $N \ge n$.
	Fix $N \ge \max\{r, s\}$.
	Let $\mm = (x_0,\ldots,x_N)$ be the graded irrelevant ideal in $R^N$.
	By assumption, we have 
	$$
	\deg_{\OO_{\PP_\kk^N}(1)}\left(s^i(Z^N, \PP_\kk^N)\right) \;=\; \deg_{\OO_{\PP_\kk^N}(1)}\left(s^i(W^N, \PP_\kk^N)\right) \quad \text{ for all } \quad i \ge 0.
	$$
	Then \autoref{thmA} yields 
	$$
	\overline{\II_{Z^N}} \;=\; \overline{\II_{W^N}} \;\subset\; \OO_{\PP_\kk^N}.
	$$
	This, in turn, leads to the following equality of saturated homogeneous ideals
	$$
	\left(\overline{IR^N} : \mm^\infty\right) \;=\; \bigoplus_{v \ge 0}\, \HH^0\big(\PP_\kk^N, \overline{\II_{Z^N}}(v)\big) \;=\; \bigoplus_{v \ge 0}\, \HH^0\big(\PP_\kk^N, \overline{\II_{W^N}}(v)\big) \;=\; \left(\overline{JR^N} : \mm^\infty\right).
	$$	
	Since $N +1 > \max\{r,s\}$, the irrelevant ideal $\mm = (x_0,\ldots,x_N)$ cannot be an associated prime of $\overline{IR^N}$ or $\overline{JR^N}$. Indeed, every associated prime of the integral closure of an ideal has height at most the minimal number of generators of the ideal (see, e.g., \cite[Proposition 3.9 and Proposition 4.1]{MCADAM}, \cite[Theorem 4.1]{PTUV}).
	It then follows that 
	$$
	\overline{IR^N} \;=\; \left(\overline{IR^N} : \mm^\infty\right) \;=\; \left(\overline{JR^N} : \mm^\infty\right) \;=\; \overline{JR^N}.
	$$
	Since the map $R \rightarrow R^N$ is faithfully flat, we have that $\overline{IR^N} \cap R = \overline{\,I\,}$ and $\overline{JR^N} \cap R = \overline{\,J\,}$ (see \cite[Proposition 1.6.2]{SwHu}).
	Finally, this implies that $\overline{\,I\,} = \overline{\,J\,}$, as required.
\end{proof}

\section*{Acknowledgments}

The author was partially supported by NSF grant DMS-2502321.
We thank the reviewer for carefully reading our paper and for several comments and corrections.

\smallskip

\noindent
\textbf{Conflict of interest statement.} The  author states that there is no conflict of interest.

\smallskip

\noindent
\textbf{Statement about data availability.} The author states that the manuscript has no associated data.
	
\bibliographystyle{amsalpha}
\bibliography{references.bib}

@Book{MATSUMURA,
  author    = {Hideyuki Matsumura},
  title     = {Commutative Ring Theory},
  publisher = {Cambridge University Press},
  year      = {1989},
  series    = {Cambridge Studies in Advanced Mathematics volume 8},
  edition   = {1},
}

@Book{HARTSHORNE,
  author    = {Hartshorne, Robin},
  title     = {Algebraic geometry},
  publisher = {Springer-Verlag, New York-Heidelberg},
  year      = {1977},
  note      = {Graduate Texts in Mathematics, No. 52},
  pages     = {xvi+496},
}

@Article{HYRY_MULTGRAD,
  author     = {Hyry, Eero},
  title      = {The diagonal subring and the {C}ohen-{M}acaulay property of a multigraded ring},
  journal    = {Trans. Amer. Math. Soc.},
  year       = {1999},
  volume     = {351},
  number     = {6},
  pages      = {2213--2232},
  issn       = {0002-9947},
  doi        = {10.1090/S0002-9947-99-02143-1},
  fjournal   = {Transactions of the American Mathematical Society},
  mrclass    = {13A30 (13D45 13H10)},
  mrnumber   = {1467469},
  mrreviewer = {Ngo Viet Trung},
  url        = {https://doi.org/10.1090/S0002-9947-99-02143-1},
}

@Book{FULTON_INTER,
  author    = {Fulton, William},
  title     = {Intersection theory},
  publisher = {Springer-Verlag, Berlin},
  year      = {1998},
  volume    = {2},
  series    = {Ergebnisse der Mathematik und ihrer Grenzgebiete. 3. Folge. A Series of Modern Surveys in Mathematics [Results in Mathematics and Related Areas. 3rd Series. A Series of Modern Surveys in Mathematics]},
  edition   = {Second},
  isbn      = {3-540-62046-X; 0-387-98549-2},
  doi       = {10.1007/978-1-4612-1700-8},
  mrclass   = {14C17 (14-02)},
  mrnumber  = {1644323},
  pages     = {xiv+470},
  url       = {https://doi.org/10.1007/978-1-4612-1700-8},
}

@misc{stacks-project,
	author       = {The {Stacks project authors}},
	title        = {The Stacks project},
	howpublished = {\url{https://stacks.math.columbia.edu}},
	year         = {2025},
}

@Book{SwHu,
  author     = {Swanson, Irena and Huneke, Craig},
  title      = {Integral closure of ideals, rings, and modules},
  publisher  = {Cambridge University Press, Cambridge},
  year       = {2006},
  volume     = {336},
  series     = {London Mathematical Society Lecture Note Series},
  isbn       = {978-0-521-68860-4; 0-521-68860-4},
  mrclass    = {13B22 (13A18 13A30 13A35 13H15 14A05)},
  mrnumber   = {2266432},
  mrreviewer = {Liam O'Carroll},
  pages      = {xiv+431},
}

@article {VanGastel,
	AUTHOR = {van Gastel, Leendert J.},
	TITLE = {Excess intersections and a correspondence principle},
	JOURNAL = {Invent. Math.},
	FJOURNAL = {Inventiones Mathematicae},
	VOLUME = {103},
	YEAR = {1991},
	NUMBER = {1},
	PAGES = {197--222},
	ISSN = {0020-9910,1432-1297},
	MRCLASS = {14C17 (14E10 14N10)},
	MRNUMBER = {1079843},
	MRREVIEWER = {Gary\ P.\ Kennedy},
	DOI = {10.1007/BF01239512},
	URL = {https://doi.org/10.1007/BF01239512},
}

@book {FOV,
	AUTHOR = {Flenner, Hubert and O'Carroll, Liam and Vogel, Wolfgang},
	TITLE = {Joins and intersections},
	SERIES = {Springer Monographs in Mathematics},
	PUBLISHER = {Springer-Verlag, Berlin},
	YEAR = {1999},
	PAGES = {vi+307},
	ISBN = {3-540-66319-3},
	MRCLASS = {14C17 (13H15 14-02)},
	MRNUMBER = {1724388},
	MRREVIEWER = {L\^e\ Tu\^an\ Hoa},
	DOI = {10.1007/978-3-662-03817-8},
	URL = {https://doi.org/10.1007/978-3-662-03817-8},
}

@Article{GG,
	author     = {Gaffney, Terence and Gassler, Robert},
	title      = {Segre numbers and hypersurface singularities},
	journal    = {J. Algebraic Geom.},
	year       = {1999},
	volume     = {8},
	number     = {4},
	pages      = {695--736},
	issn       = {1056-3911},
	fjournal   = {Journal of Algebraic Geometry},
	mrclass    = {32S15 (13H15 32S10 32S25)},
	mrnumber   = {1703611},
	mrreviewer = {Aleksandr G. Aleksandrov},
}

@article{CRPU2,
	title={Multidegrees, families, and integral dependence},
	author={Cid-Ruiz, Yairon and Polini, Claudia and Ulrich, Bernd},
	journal={arXiv preprint arXiv:2405.07000},
	year={2024}
}

@Article{KLEIMAN_THORUP_GEOM,
	author     = {Kleiman, Steven and Thorup, Anders},
	title      = {A geometric theory of the {B}uchsbaum-{R}im multiplicity},
	journal    = {J. Algebra},
	year       = {1994},
	volume     = {167},
	number     = {1},
	pages      = {168--231},
	issn       = {0021-8693},
	doi        = {10.1006/jabr.1994.1182},
	fjournal   = {Journal of Algebra},
	mrclass    = {14C17 (13D40 13H15)},
	mrnumber   = {1282823},
	mrreviewer = {R\"{u}diger Achilles},
	url        = {https://doi-org.kuleuven.e-bronnen.be/10.1006/jabr.1994.1182},
}

@Article{KLEIMAN_THORUP_MIXED,
	author   = {Kleiman, Steven and Thorup, Anders},
	title    = {Mixed {B}uchsbaum-{R}im multiplicities},
	journal  = {Amer. J. Math.},
	year     = {1996},
	volume   = {118},
	number   = {3},
	pages    = {529--569},
	issn     = {0002-9327},
	fjournal = {American Journal of Mathematics},
	mrclass  = {14C17 (13H15)},
	mrnumber = {1393259},
	url      = {http://muse.jhu.edu.kuleuven.e-bronnen.be/journals/american_journal_of_mathematics/v118/118.3kleiman.pdf},
}

@incollection {KLEIMAN_PICARD,
	AUTHOR = {Kleiman, Steven L.},
	TITLE = {The {P}icard scheme},
	BOOKTITLE = {Fundamental algebraic geometry},
	SERIES = {Math. Surveys Monogr.},
	VOLUME = {123},
	PAGES = {235--321},
	PUBLISHER = {Amer. Math. Soc., Providence, RI},
	YEAR = {2005},
	ISBN = {0-8218-3541-6},
	MRCLASS = {14C22},
	MRNUMBER = {2223410},
}

@article {KLEIMAN_NUM,
	AUTHOR = {Kleiman, Steven L.},
	TITLE = {Toward a numerical theory of ampleness},
	JOURNAL = {Ann. of Math. (2)},
	FJOURNAL = {Annals of Mathematics. Second Series},
	VOLUME = {84},
	YEAR = {1966},
	PAGES = {293--344},
	ISSN = {0003-486X},
	MRCLASS = {14.55 (14.10)},
	MRNUMBER = {206009},
	DOI = {10.2307/1970447},
	URL = {https://doi-org.prox.lib.ncsu.edu/10.2307/1970447},
}

@article {SNAPPER,
	AUTHOR = {Snapper, Ernst},
	TITLE = {Multiples of divisors},
	JOURNAL = {J. Math. Mech.},
	FJOURNAL = {J. Math. Mech.},
	VOLUME = {8},
	YEAR = {1959},
	PAGES = {967--992},
	MRCLASS = {14.00},
	MRNUMBER = {109156},
	MRREVIEWER = {M.\ Nagata},
	DOI = {10.1512/iumj.1959.8.58062},
	URL = {https://doi-org.prox.lib.ncsu.edu/10.1512/iumj.1959.8.58062},
}

@article {LJT,
	AUTHOR = {Lejeune-Jalabert, Monique and Teissier, Bernard},
	TITLE = {Cl\^oture int\'egrale des id\'eaux et \'equisingularit\'e},
	NOTE = {With an appendix by Jean-Jacques Risler},
	JOURNAL = {Ann. Fac. Sci. Toulouse Math. (6)},
	FJOURNAL = {Annales de la Facult\'e{} des Sciences de Toulouse.
	Math\'ematiques. S\'erie 6},
	VOLUME = {17},
	YEAR = {2008},
	NUMBER = {4},
	PAGES = {781--859},
	ISSN = {0240-2963,2258-7519},
	MRCLASS = {32S15 (14B05 14P15)},
	MRNUMBER = {2499856},
	MRREVIEWER = {Marcelo\ Jos\'e\ Saia},
	URL = {http://afst.cedram.org/item?id=AFST_2008_6_17_4_781_0},
}

@book {LAZARSFELD2,
	AUTHOR = {Lazarsfeld, Robert},
	TITLE = {Positivity in algebraic geometry. {II}},
	SERIES = {Ergebnisse der Mathematik und ihrer Grenzgebiete. 3. Folge. A
	Series of Modern Surveys in Mathematics [Results in
	Mathematics and Related Areas. 3rd Series. A Series of Modern
	Surveys in Mathematics]},
	VOLUME = {49},
	NOTE = {Positivity for vector bundles, and multiplier ideals},
	PUBLISHER = {Springer-Verlag, Berlin},
	YEAR = {2004},
	PAGES = {xviii+385},
	ISBN = {3-540-22534-X},
	MRCLASS = {14-02 (14C20 14F05 14F17)},
	MRNUMBER = {2095472},
	MRREVIEWER = {Mihnea\ Popa},
	DOI = {10.1007/978-3-642-18808-4},
	URL = {https://doi-org.prox.lib.ncsu.edu/10.1007/978-3-642-18808-4},
}

@book {VASC_INT,
	AUTHOR = {Vasconcelos, Wolmer},
	TITLE = {Integral closure},
	SERIES = {Springer Monographs in Mathematics},
	NOTE = {Rees algebras, multiplicities, algorithms},
	PUBLISHER = {Springer-Verlag, Berlin},
	YEAR = {2005},
	PAGES = {xii+519},
	ISBN = {978-3-540-25540-6; 3-540-25540-0},
	MRCLASS = {13A30 (13-02 13B22 13D40 13H15)},
	MRNUMBER = {2153889},
	MRREVIEWER = {Ngo Viet Trung},
}

@article {aluffi2017segre,
	AUTHOR = {Aluffi, Paolo},
	TITLE = {The {S}egre zeta function of an ideal},
	JOURNAL = {Adv. Math.},
	FJOURNAL = {Advances in Mathematics},
	VOLUME = {320},
	YEAR = {2017},
	PAGES = {1201--1226},
	ISSN = {0001-8708,1090-2082},
	MRCLASS = {14C17 (11M41 14Q15)},
	MRNUMBER = {3709134},
	MRREVIEWER = {Zach\ Teitler},
	DOI = {10.1016/j.aim.2017.09.023},
	URL = {https://doi.org/10.1016/j.aim.2017.09.023},
}

@article {aluffi2024lorentzian,
	AUTHOR = {Aluffi, Paolo},
	TITLE = {Lorentzian polynomials, {S}egre classes, and adjoint
	polynomials of convex polyhedral cones},
	JOURNAL = {Adv. Math.},
	FJOURNAL = {Advances in Mathematics},
	VOLUME = {437},
	YEAR = {2024},
	PAGES = {Paper No. 109440, 37},
	ISSN = {0001-8708,1090-2082},
	MRCLASS = {14C17 (52A20)},
	MRNUMBER = {4674860},
	MRREVIEWER = {I.\ Dolgachev},
	DOI = {10.1016/j.aim.2023.109440},
	URL = {https://doi.org/10.1016/j.aim.2023.109440},
}

@Article{UV_CRIT_MOD,
	author     = {Ulrich, Bernd and Validashti, Javid},
	journal    = {Math. Res. Lett.},
	title      = {A criterion for integral dependence of modules},
	year       = {2008},
	issn       = {1073-2780},
	number     = {1},
	pages      = {149--162},
	volume     = {15},
	doi        = {10.4310/MRL.2008.v15.n1.a13},
	fjournal   = {Mathematical Research Letters},
	mrclass    = {13C15 (13B21 13B22)},
	mrnumber   = {2367181},
	mrreviewer = {Aron Simis},
	url        = {https://doi.org/10.4310/MRL.2008.v15.n1.a13},
}

@Article{UV_NUM_CRIT,
	author     = {Ulrich, Bernd and Validashti, Javid},
	journal    = {Math. Proc. Cambridge Philos. Soc.},
	title      = {Numerical criteria for integral dependence},
	year       = {2011},
	issn       = {0305-0041},
	number     = {1},
	pages      = {95--102},
	volume     = {151},
	doi        = {10.1017/S0305004111000144},
	fjournal   = {Mathematical Proceedings of the Cambridge Philosophical Society},
	mrclass    = {13A30 (13A02)},
	mrnumber   = {2801316},
	mrreviewer = {Florian Enescu},
	url        = {https://doi.org/10.1017/S0305004111000144},
}

@InCollection{GAFFNEY,
	author     = {Gaffney, Terence},
	title      = {Generalized {B}uchsbaum-{R}im multiplcities and a theorem of {R}ees},
	year       = {2003},
	note       = {Special issue in honor of Steven L. Kleiman},
	number     = {8},
	pages      = {3811--3827},
	volume     = {31},
	doi        = {10.1081/AGB-120022444},
	fjournal   = {Communications in Algebra},
	issn       = {0092-7872},
	journal    = {Comm. Algebra},
	mrclass    = {13H15 (13B22)},
	mrnumber   = {2007386},
	mrreviewer = {Marcel Morales},
	url        = {https://doi.org/10.1081/AGB-120022444},
}

@Article{FLENNER_MANARESI,
	author     = {Flenner, Hubert and Manaresi, Mirella},
	journal    = {Math. Z.},
	title      = {A numerical characterization of reduction ideals},
	year       = {2001},
	issn       = {0025-5874},
	number     = {1},
	pages      = {205--214},
	volume     = {238},
	doi        = {10.1007/PL00004900},
	fjournal   = {Mathematische Zeitschrift},
	mrclass    = {13H15 (13A30)},
	mrnumber   = {1860742},
	mrreviewer = {Keri Sather-Wagstaff},
	url        = {https://doi.org/10.1007/PL00004900},
}

@Article{BOEGER,
	author     = {B\"{o}ger, Erwin},
	journal    = {Math. Ann.},
	title      = {Einige {B}emerkungen zur {T}heorie der ganz-algebraischen {A}bh\"{a}ngigkeit von {I}dealen},
	year       = {1970},
	issn       = {0025-5831},
	pages      = {303--308},
	volume     = {185},
	doi        = {10.1007/BF01349952},
	fjournal   = {Mathematische Annalen},
	mrclass    = {13.95 (14.18)},
	mrnumber   = {263809},
	mrreviewer = {H. Yanagihara},
	url        = {https://doi.org/10.1007/BF01349952},
}

@Article{REES,
	author     = {Rees, David},
	journal    = {Proc. Cambridge Philos. Soc.},
	title      = {{${\mathfrak{a}}$}-transforms of local rings and a theorem on multiplicities of ideals},
	year       = {1961},
	issn       = {0008-1981},
	pages      = {8--17},
	volume     = {57},
	doi        = {10.1017/s0305004100034800},
	fjournal   = {Proceedings of the Cambridge Philosophical Society},
	mrclass    = {16.00},
	mrnumber   = {118750},
	mrreviewer = {H.\ T.\ Muhly},
	url        = {https://doi.org/10.1017/s0305004100034800},
}

@InCollection{TEISSIER_CYC,
	author     = {Teissier, Bernard},
	booktitle  = {Singularit\'{e}s \`a {C}arg\`ese ({R}encontre {S}ingularit\'{e}s {G}\'{e}om. {A}nal., {I}nst. \'{E}tudes {S}ci., {C}arg\`ese, 1972)},
	publisher  = {Soc. Math. France, Paris},
	title      = {Cycles \'{e}vanescents, sections planes et conditions de {W}hitney},
	year       = {1973},
	pages      = {285--362},
	series     = {Ast\'{e}risque},
	volume     = {Nos. 7 et 8},
	mrclass    = {32C40},
	mrnumber   = {374482},
	mrreviewer = {J.\ A.\ Morrow},
}

@Article{TEISSIER_RES2,
	author    = {Teissier, Bernard},
	journal   = {Séminaire sur les singularités des surfaces},
	title     = {Résolution simultanée : {II} - Résolution simultanée et cycles évanescents},
	year      = {1976-1977},
	pages     = {1-66},
	language  = {fre},
	publisher = {Ecole Polytechnique, Centre de Mathématiques},
	url       = {http://eudml.org/doc/114142},
}

@Book{MCADAM,
	author     = {McAdam, Stephen},
	publisher  = {Springer-Verlag, Berlin},
	title      = {Asymptotic prime divisors},
	year       = {1983},
	isbn       = {3-540-12722-4},
	series     = {Lecture Notes in Mathematics},
	volume     = {1023},
	doi        = {10.1007/BFb0071575},
	mrclass    = {13E05 (13A17 13B20 13C15 13G05 13H99)},
	mrnumber   = {722609},
	mrreviewer = {David\ E.\ Dobbs},
	pages      = {ix+118},
	url        = {https://doi.org/10.1007/BFb0071575},
}

@Article{PTUV,
	author     = {Polini, Claudia and Trung, Ngo Viet and Ulrich, Bernd and Validashti, Javid},
	title      = {Multiplicity sequence and integral dependence},
	journal    = {Math. Ann.},
	year       = {2020},
	volume     = {378},
	number     = {3-4},
	pages      = {951--969},
	issn       = {0025-5831},
	doi        = {10.1007/s00208-020-02059-5},
	fjournal   = {Mathematische Annalen},
	mrclass    = {13B22 (13A30 13D40 14B05)},
	mrnumber   = {4163518},
	mrreviewer = {Catalin Ciuperca},
	url        = {https://doi.org/10.1007/s00208-020-02059-5},
}

@article {JORGENSON,
	AUTHOR = {Jorgenson, Grayson},
	TITLE = {A relative {S}egre zeta function},
	JOURNAL = {J. Pure Appl. Algebra},
	FJOURNAL = {Journal of Pure and Applied Algebra},
	VOLUME = {224},
	YEAR = {2020},
	NUMBER = {12},
	PAGES = {106437, 15},
	ISSN = {0022-4049,1873-1376},
	MRCLASS = {14C17},
	MRNUMBER = {4099922},
	MRREVIEWER = {Xia\ Liao},
	DOI = {10.1016/j.jpaa.2020.106437},
	URL = {https://doi-org.prox.lib.ncsu.edu/10.1016/j.jpaa.2020.106437},
}

@article{MIXED_SEG_ZETA,
	title={Mixed Segre zeta functions and their log-concavity},
	author={Cid-Ruiz, Yairon},
	journal={arXiv preprint arXiv:2507.06424},
	year={2025}
}

@article {cidruiz2024polar,
	AUTHOR = {Cid-Ruiz, Yairon},
	TITLE = {Polar multiplicities and integral dependence},
	JOURNAL = {Int. Math. Res. Not. IMRN},
	FJOURNAL = {International Mathematics Research Notices. IMRN},
	YEAR = {2024},
	NUMBER = {17},
	PAGES = {12201--12218},
	ISSN = {1073-7928,1687-0247},
	MRCLASS = {13H15 (13B21)},
	MRNUMBER = {4795000},
	MRREVIEWER = {Yu\ Xie},
	DOI = {10.1093/imrn/rnae163},
	URL = {https://doi-org.prox.lib.ncsu.edu/10.1093/imrn/rnae163},
}

@Article{SUV_MULT,
	author   = {Simis, Aron and Ulrich, Bernd and Vasconcelos, Wolmer V.},
	journal  = {Math. Proc. Cambridge Philos. Soc.},
	title    = {Codimension, multiplicity and integral extensions},
	year     = {2001},
	number   = {2},
	pages    = {237--257},
	volume   = {130},
	fjournal = {Mathematical Proceedings of the Cambridge Philosophical Society},
}

@article{cid2023relative,
	title={Relative mixed multiplicities and mixed {B}uchsbaum-{R}im multiplicities},
	author={Cid-Ruiz, Yairon},
	journal={arXiv preprint arXiv:2311.15105},
	year={2023}
}

@incollection {ALUFFI_SURVEY,
	AUTHOR = {Aluffi, Paolo},
	TITLE = {Segre classes and invariants of singular varieties},
	BOOKTITLE = {Handbook of geometry and topology of singularities {III}},
	PAGES = {419--492},
	PUBLISHER = {Springer, Cham},
	YEAR = {[2022] \copyright 2022},
	ISBN = {978-3-030-95759-9; 978-3-030-95760-5},
	MRCLASS = {14C17},
	MRNUMBER = {4461013},
	DOI = {10.1007/978-3-030-95760-5\_6},
	URL = {https://doi-org.prox.lib.ncsu.edu/10.1007/978-3-030-95760-5_6},
}

@article{das2024numerical,
	title={Numerical characterizations for integral dependence of graded ideals},
	author={Das, Suprajo and Roy, Sudeshna and Trivedi, Vijaylaxmi},
	journal={arXiv preprint arXiv:2409.09346},
	year={2024}
}

@book {VOGEL,
	AUTHOR = {Vogel, Wolfgang},
	TITLE = {Lectures on results on {B}ezout's theorem},
	SERIES = {Tata Institute of Fundamental Research Lectures on Mathematics
	and Physics},
	VOLUME = {74},
	NOTE = {Notes by D. P. Patil},
	PUBLISHER = {Tata Institute of Fundamental Research, Bombay; by
	Springer-Verlag, Berlin},
	YEAR = {1984},
	PAGES = {vi+132},
	ISBN = {3-540-12679-1},
	MRCLASS = {14C17 (13H15)},
	MRNUMBER = {743265},
	MRREVIEWER = {Manfred\ Herrmann},
	DOI = {10.1007/978-3-662-00493-7},
	URL = {https://doi-org.prox.lib.ncsu.edu/10.1007/978-3-662-00493-7},
}

@article {ACH_MANA,
	AUTHOR = {Achilles, R\"{u}diger and Manaresi, Mirella},
	TITLE = {Multiplicities of a bigraded ring and intersection theory},
	JOURNAL = {Math. Ann.},
	FJOURNAL = {Mathematische Annalen},
	VOLUME = {309},
	YEAR = {1997},
	NUMBER = {4},
	PAGES = {573--591},
	ISSN = {0025-5831,1432-1807},
	MRCLASS = {14C17 (13A30 13H15)},
	MRNUMBER = {1483824},
	MRREVIEWER = {Marc\ Chardin},
	DOI = {10.1007/s002080050128},
	URL = {https://doi-org.prox.lib.ncsu.edu/10.1007/s002080050128},
}

@article {BLQ,
	AUTHOR = {Blum, Harold and Liu, Yuchen and Qi, Lu},
	TITLE = {Convexity of multiplicities of filtrations on local rings},
	JOURNAL = {Compos. Math.},
	FJOURNAL = {Compositio Mathematica},
	VOLUME = {160},
	YEAR = {2024},
	NUMBER = {4},
	PAGES = {878--914},
	ISSN = {0010-437X,1570-5846},
	MRCLASS = {14B05 (13H15)},
	MRNUMBER = {4716588},
	MRREVIEWER = {Guillaume\ Rond},
	DOI = {10.1112/S0010437X23007972},
	URL = {https://doi-org.prox.lib.ncsu.edu/10.1112/S0010437X23007972},
}

@article {CIU,
	AUTHOR = {Ciuperc\u{a}, C\u{a}t\u{a}lin},
	TITLE = {A numerical characterization of the {$S_2$}-ification of a
	{R}ees algebra},
	JOURNAL = {J. Pure Appl. Algebra},
	FJOURNAL = {Journal of Pure and Applied Algebra},
	VOLUME = {178},
	YEAR = {2003},
	NUMBER = {1},
	PAGES = {25--48},
	ISSN = {0022-4049,1873-1376},
	MRCLASS = {13H15 (13A30 13D40)},
	MRNUMBER = {1947965},
	MRREVIEWER = {Maria\ Evelina\ Rossi},
	DOI = {10.1016/S0022-4049(02)00157-3},
	URL = {https://doi-org.prox.lib.ncsu.edu/10.1016/S0022-4049(02)00157-3},
}
	
\end{document}